%% file: Eaton.tex
\documentclass[a4paper]{amsart}
\usepackage[T1]{fontenc}
\usepackage[utf8]{inputenc}
\usepackage[english]{babel}
\usepackage[sc]{mathpazo}
\usepackage{mathrsfs}			% for mathscr
\usepackage{mathtools}			% for norms and absolute values
	\mathtoolsset{centercolon}	% better looking :=
\usepackage{braket}			% \set has fixed dimensions, \Set adjustable ones
\usepackage{enumerate}
\usepackage{comment}
\usepackage{import}			% drawings
\usepackage{microtype}
\usepackage[autostyle,english=british]{csquotes}
\usepackage[style=numeric-comp,backend=bibtex]{biblatex}
	\renewbibmacro{in:}{\ifentrytype{article}{}{\printtext{\bibstring{in}\intitlepunct}}}
\usepackage[colorlinks]{hyperref}
\theoremstyle{plain}
	\newtheorem{thm}{Theorem}
	\newtheorem{cor}[thm]{Corollary}
	\newtheorem{lemma}[thm]{Lemma}
	\newtheorem{prop}[thm]{Proposition}
\theoremstyle{definition}
	\newtheorem{definition}[thm]{Definition}
\theoremstyle{remark}
	
	\newtheorem*{condition}{Flat Admissibility Condition}
	\newtheorem*{convention}{Convention}

\newcommand{\numberset}{\mathbb}
\newcommand{\CC}{\numberset{C}}
\newcommand{\NN}{\numberset{N}}
\newcommand{\PP}{\numberset{P}}
\newcommand{\RR}{\numberset{R}}
\newcommand{\TT}{\numberset{T}}
\newcommand{\ZZ}{\numberset{Z}}

\newcommand{\cH}{\mathcal{H}}
\newcommand{\cL}{\mathcal{L}}

\newcommand{\sL}{\mathscr{L}}

\renewcommand{\epsilon}{\varepsilon}
\renewcommand{\phi}{\varphi}

\DeclareMathOperator{\GL}{GL}
\DeclareMathOperator{\PGL}{PGL}
\DeclareMathOperator{\SL}{SL}
\DeclareMathOperator{\PSL}{PSL}

\DeclareMathOperator{\hol}{hol}
\DeclareMathOperator{\id}{id}
\DeclareMathOperator{\Leb}{Leb}

\DeclarePairedDelimiter{\abs}{\lvert}{\rvert}
\DeclarePairedDelimiter{\norm}{\lVert}{\rVert}

\newcommand{\blank}{\,{\cdot}\,}
\newcommand{\flow}{\phi_{t}^{\theta}}
\newcommand{\tflow}{\widetilde{\phi}_{t}^{\theta}}

\newcommand{\rot}[1]{\begin{pmatrix}\cos{#1}&-\sin{#1}\\ \sin{#1}&\cos{#1}\end{pmatrix}}

\begin{document}

% Data
\title{Exceptional ergodic directions in Eaton lenses}
\author{Mauro Artigiani}
\thanks{Research supported by a University of Bristol graduate scholarship}
\address{School of Mathematics \\ University of Bristol \\ University Walk \\ Bristol \\ BS8~1TW \\ United Kingdom}
\email{mauro.artigiani@bristol.ac.uk}

%Document
\begin{abstract}
We construct examples of ergodic vertical flows in periodic configurations of Eaton lenses of fixed radius.
We achieve this by studying a family of infinite translation surfaces that are $\ZZ^2$-covers of slit tori.
We show that the Hausdorff dimension of lattices for which the vertical flow is ergodic is bigger than $3/2$.
Moreover, the lattices are explicitly constructed. 
\end{abstract}

\maketitle
\section{Introduction}
\subsection{Setting and statement of the results}
Circular Eaton lenses in the plane $\RR^2$ were introduced in~\cite{Eaton} as an example of a perfect retro-flector: when a light ray enters a lens it is reflected in the same direction with opposite orientation, see the left part of Figure~\ref{fig:onelens}.
We consider a system of such lenses of some fixed radius $R>0$ whose centres are placed on a lattice $\Lambda\subset\RR^2$, as recently studied by K.~Fr\k{a}czek and M.~Schmoll in~\cite{FraczekSchmoll}.
This leads to the study of an infinite periodic billiard, an area that has been intensively studied in the last few years; see Section~\ref{sec:literature} for more details and references.

\begin{figure}[ht]
\centering
\includegraphics[width=0.6\textwidth]{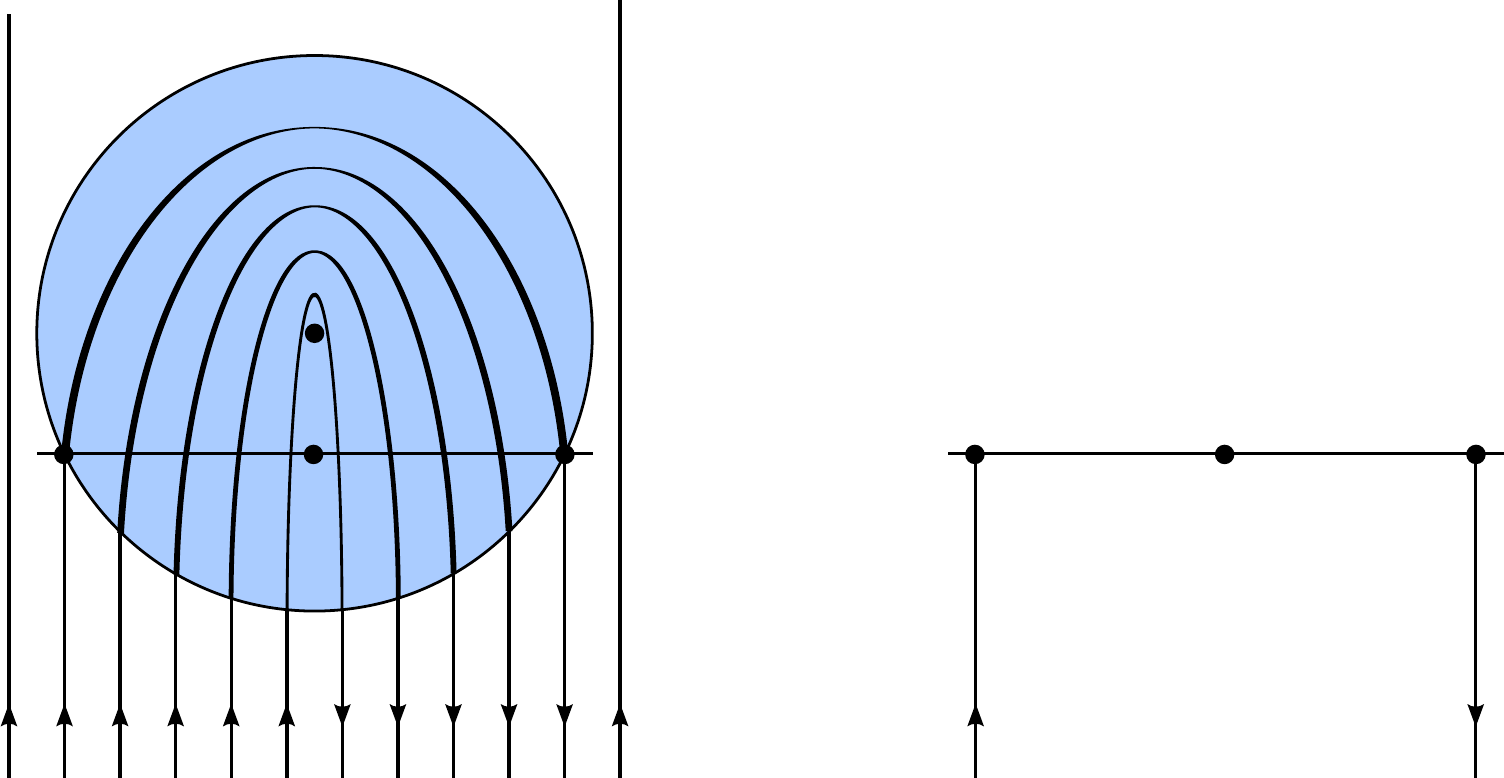}
\caption{Vertical trajectories entering a circular Eaton lens and the flat counterpart.}
\label{fig:onelens}
\end{figure}

A lattice $\Lambda$ is \emph{$R$-admissible} if the circles of radius $R$ centred at the lattice points do not overlap.
A system of Eaton lenses will be denoted $L(\Lambda,R)$.
Applying an appropriate rotation to the lattice, we can restrict ourselves to the study of vertical light rays.
Moreover, up to scaling, we can suppose that the lattice $\Lambda$ has covolume $1$, in other words $\Lambda\in\sL=\SL(2,\RR)/\SL(2,\ZZ)$.
From the circle packing problem, one knows that if $R\geq1/\sqrt{2\sqrt{3}}$ there are no lattices of covolume one that are $R$-admissible, while for lower values of the radius the set of $R$-admissible lattices is a non-empty open set in $\sL$.
We denote with $\mu_\sL$ the unique probability measure on $\sL$ which is invariant by the action by left multiplication by elements of $\SL(2,\RR)$.
K.~Fr\k{a}czek and M.~Schmoll recently discovered in~\cite{FraczekSchmoll} an interesting behaviour of light rays in Eaton lenses.
They proved the following

\begin{thm}
For every $0<R<1/\sqrt{2\sqrt{3}}$ and for $\mu_\sL$-almost every $R$-admissible lattice $\Lambda\in\sL$ there exist constants $C=C(\Lambda,R)>0$ and $\theta=\theta(\Lambda,R)\in S^1$ such that every vertical light ray in $L(\Lambda,R)$ is trapped in an infinite band of width $C>0$ in direction $\theta$.
\end{thm}

A natural question is thus if there are \emph{exceptional} cases, that is if, for every radius, one can find lattices in which the vertical light rays are not confined.
In this paper we explicitly construct exceptional lattices and we give a lower bound on the Hausdorff dimension of the set of exceptional lattices, thus showing that this set is rich in the measure-theoretic sense.
We define the \emph{Eaton flow} as the flow that moves every point of the plane vertically with unit speed following trajectories of light rays.
This flow preserves the Lebesgue measure $\Leb$ on $\RR^2$.
We recall that a flow is ergodic if every set $A\subset\RR^2$ that is invariant under the flow has either $\Leb (A)=0$ or $\Leb (\RR^2\setminus A)=0$.
In particular, since the Lebesgue measure gives positive measure to all open sets, almost every orbit under the flow is \emph{dense}, and hence far from being trapped in a strip.
The main result of this paper is the following

\begin{thm}\label{thm:HausdorffEaton}
Let $0<R<1/2$.
Then there exists a set of $R$-admissible lattices $\Lambda\in\sL$ whose Hausdorff dimension is bigger than $3/2$ such that the Eaton flow is ergodic with respect to the Lebesgue measure on $L(\Lambda,R)$.
\end{thm}

In order to prove this result we study a related system obtained by replacing each circular lens with a horizontal obstacle of the same length of the diameter of the lens, centred in the middle point of the lens itself.
When a vertical light ray encounters any of these obstacles it is rotated by $180$ degrees around the centre of the obstacle and comes out with the opposite orientation, see the right part of Figure~\ref{fig:onelens}.
If a ray hits the centre of the obstacle, by convention we prolong the orbit on the same line with reversed orientation. 
We denote a system of these ``flat lenses'' with $F(\Lambda,R)$.
It follows from the construction that the orbits of vertical light rays are the same in $L(\Lambda,R)$ and in $F(\Lambda,R)$ except inside of each circular lens, see Figure~\ref{fig:morelenses}.
In particular, for the study of ergodicity, we can use the simpler system $F(\Lambda,R)$ to deduce information on our original setting of Eaton lenses.

\begin{figure}[bt]
\centering
\def\svgwidth{0.8\textwidth}
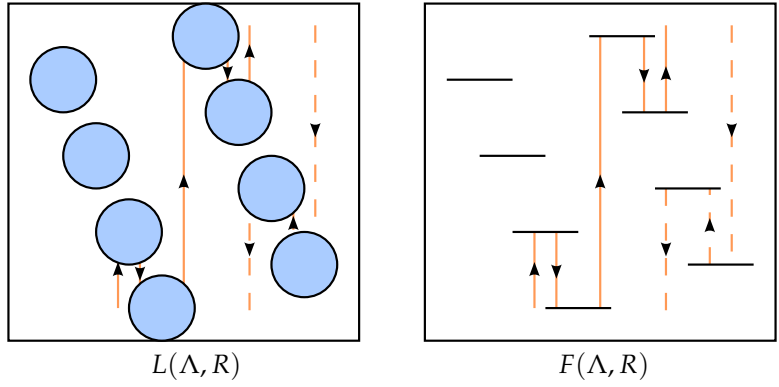
\caption{Trajectories in a periodic configuration of Eaton lenses $L(\Lambda,R)$ and their counterparts in the flat lenses system $F(\Lambda,R)$.}
\label{fig:morelenses}
\end{figure}

We will always assume the following

\begin{condition}
For $R>0$ and $\Lambda\in\sL$ the obstacles in $F(\Lambda,R)$ are pairwise disjoint.
\end{condition}

In particular, if $\Lambda$ is $R$-admissible, then the system $F(\Lambda,R)$ satisfies the flat admissibility condition.
On the infinite surface $F(\Lambda,R)$ the vertical trajectories of the light rays give rise to a non-orientable vertical foliation.
However, one can obtain an orientable foliation constructing a double cover $\widetilde{M}(\Lambda,R)$, called the \emph{orientation covering} of $F(\Lambda,R)$, in the following way.
Take two copies $F_\pm(\Lambda,R)$ of $F(\Lambda,R)$, corresponding to the two possible orientations of a vertical trajectory.
Every light ray travels in one copy $F_\pm(\Lambda,R)$ until it hits one obstacle; it is then rotated by $180$ degrees around the centre of the obstacle and comes out in the opposite copy $F_\mp(\Lambda,R)$, see Figure~\ref{fig:flatlenses}.
Denote $r_\pi F_-(\Lambda,R)$ the image of $F_-(\Lambda,R)$ under the rotation by $180$ degrees around the origin.
We enumerate the obstacles in $F_\pm(\Lambda,R)$ in the obvious way with elements of $\ZZ^2$.
Then $r_\pi F_-(\Lambda, R)$ inherits an enumeration of its obstacles from the one given to $F_-(\Lambda,R)$.
We obtain the surface $\widetilde{M}(\Lambda,R)$ by gluing the left (resp.\ right) part of the obstacle numbered by $(m,n)$ of $F_+(\Lambda,R)$ to the right (resp.\ left) part of the obstacle numbered by $(m,n)$ of $r_\pi F_-(\Lambda,R)$.
Then $\widetilde{M}(\Lambda,R)$ is a translation surface in which the vertical foliation becomes orientable.
We define the vertical directional flow $\widetilde{\phi}_t^v$ as the flow that moves up, at unit speed, points along the leaves of the orientable foliation we obtained on $\widetilde{M}(\Lambda,R)$.
Choosing a fundamental domain for the $\Lambda$-action on $\widetilde{M}(\Lambda,R)$ one sees that this surface is a $\ZZ^2$-cover of a \emph{compact} translation surface, denoted $M(\Lambda,R)$, given by two flat tori glued along a slit, that comes from the obstacles in $F(\Lambda,R)$.

\begin{figure}[bt]
\centering
\def\svgwidth{0.95\textwidth}
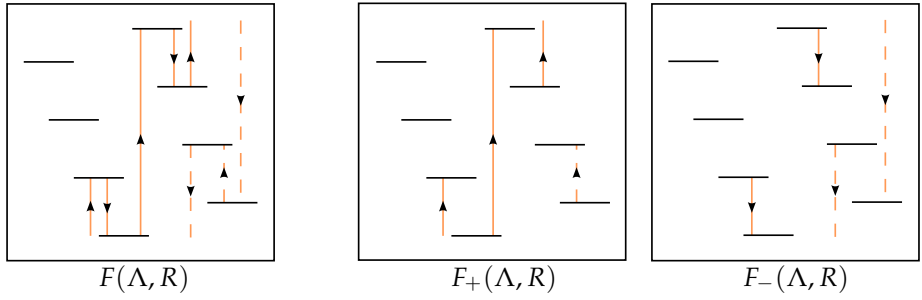
\caption{Two trajectories in a periodic configuration of flat lenses $F(\Lambda,R)$ and their images in $F_\pm(\Lambda,R)$.}
\label{fig:flatlenses}
\end{figure}

To prove our main result, we will study in depth the case when the two tori are obtained from the unit square in $\RR^2$, that is the case when $\Lambda=\ZZ^2$, and prove a result analogous to Theorem~\ref{thm:HausdorffEaton} for the $\ZZ^2$-cover $\widetilde{M}(\ZZ^2,R)$ of $M(\ZZ^2,R)$.
Namely we prove the following

\begin{thm}\label{thm:MR}
Let $0<R<1/2$ be a rational number with odd numerator.
Then there exists a set of directions $\theta$, \emph{explicitly given in terms of their continued fraction expansions}, with Hausdorff dimension bigger than $1/2$ such that the flow $\widetilde{\phi}_t^\theta$ is \emph{ergodic} on $\widetilde{M}(\ZZ^2,R)$.
Moreover, there exists a $G_\delta$ dense subset of $S^1$ on which the same happens.
\end{thm}

The forementioned Theorem~\ref{thm:HausdorffEaton} will then follow from this result by exploiting the action of $\SL(2,\RR)$ on $\sL$ by left multiplication, as we will explain in the last section.

\subsection{Outline of the paper}
We will recall all the basic definitions of translation surfaces and their coverings in section~\ref{sec:background}.
We then state a criterion for ergodicity on $\ZZ^2$-covers of compact translation surfaces, which we will use to produce ergodic directions for $\widetilde{M}(\ZZ^2,R)$.
In section~\ref{sec:ergodicdirections} we exploit this criterion under some hypothesis on the action of $\SL(2,\RR)$ on the homology of the compact surface $M(\ZZ^2,R)$.
This action is studied in detail in the following section to guarantee that our previous hypothesis are satisfied.
We summarise our results for $\widetilde{M}(\ZZ^2,R)$ and prove a more precise version of Theorem~\ref{thm:MR} in section~\ref{sec:proofMR}.
Finally, in section~\ref{sec:returntoeaton}, we prove Theorem~\ref{thm:HausdorffEaton} using the results of the previous sections.

\subsection{Relation with other results in the literature}\label{sec:literature}
As we mentioned above, the study of $\ZZ^d$-covers of compact translation surfaces has been a very active area of research in the last few years.
The reduction of periodic systems of identical Eaton lenses to this framework was explained by K.~Fr\k{a}czek and M.~Schmoll in~\cite{FraczekSchmoll}.
Our work is inspired by the approach of K.~Fr\k{a}czek and C.~Ulcigrai in~\cite{FraczekUlcigrai:tube}, where a similar strategy is used for a $\ZZ$-periodic infinite surface.

One could see the infinite translation surface $\widetilde{M}(\ZZ^2,R)$ as a degenerate case of the famous wind-tree model, that has been intensively studied recently, for instance in~\cite{AvilaHubert:recurrence,Delecroix:divergent,DelecroixZorich:Ehrenfest,FraczekUlcigrai:notergodic,Hooper:IET,HLT:Ehrenfest}.
Also for this model the generic behaviour of the directional flow is not ergodic, as shown in~\cite{FraczekUlcigrai:notergodic}.
It is worth to remark that in the wind-tree model the trajectories are not trapped in bands and hence there is no geometrically clear picture of the non-ergodic behaviour of the directional flow.
On the technical side the proof in~\cite{FraczekSchmoll} exploits the \emph{bounded deviations} phenomenon, discovered by A.~Zorich, which is not relevant for the wind-tree as the curves giving the $\ZZ^2$-cover in this case belong to two different blocks of the so-called \emph{Kontsevich--Zorich cocycle}.
 
Ergodic directions for the directional flow on the surface $\widetilde{M}(z)_\Gamma$, introduced in Section~\ref{sec:ergodicdirections}, can be obtained using the much more general work of P.~Hooper~\cite{Hooper:IET} in some specific examples, for instance when $z=\bigl(\frac{1}{4},0\bigr)$.
However, our approach is quite different.

Finally, let us point out that we study in detail a surface $M(z)$, made out of two identical tori glued along a slit, see Section~\ref{sec:ergodicdirections}.
This surface has a rich history, going back to the seminal work of W.~Veech.
For instance, it was used to produce examples of minimal but non-ergodic directions on compact translation surfaces, see the survey of H.~Masur and S.~Tabachnikov~\cite{MasurTabachnikov} for more details.
Moreover, it plays a fundamental role in~\cite{FraczekUlcigrai:tube}.

\section{Background}\label{sec:background}

\subsection{Translation surfaces}
In this section we recall the basic definitions related with compact translation surfaces and their $\ZZ^d$-coverings.
For more details on the compact case we refer to the surveys~\cite{Masur:ergodicsurfaces,MasurTabachnikov,Viana:ietf,Zorich:flat}.

A translation surface is a pair $(M,\omega)$, where $M$ is a compact Riemann surface and $\omega$ is a nonzero Abelian differential, that is a holomorphic 1-form.
Call $\Sigma\subset M$ the set of zeros of $\omega$.
These points are the \emph{singularities} of the translation surface.
For every angle $\theta\in S^1$ one defines a vector field $X_\theta^\omega$ in direction $\theta$ on the complement $M\setminus\Sigma$ of the singularities by $\omega(X_\theta^\omega)=e^{i\theta}$.
The corresponding flow will be denoted $\flow$ and is called directional flow or straight-line flow.
This flow preserves the natural area form on $M$ given by $\frac{i}{2}\omega\wedge\overline{\omega}$.
The total area of the surface, with respect to this area form, is denoted $A(\omega)$.

A \emph{saddle connection} on $M$ is a geodesic segment for the natural flat metric of the surface that connects two singularities, not necessarily distinct, and without any other singularity in its interior.
To each curve $\gamma$ we can associate a \emph{displacement} (or holonomy) \emph{vector} obtained developing the curve from $M$ to $\RR^2$ and then taking the difference between the final and initial points on the Euclidean geodesic.
Identifying $\RR^2$ with $\CC$ one has $\hol(\gamma)=\int_\gamma \omega$.
A \emph{cylinder} $C\subset M$ is a maximal connected union of simple closed geodesics all of which are homotopic one to the other.
A closed geodesic in a cylinder is called a \emph{core curve} of the cylinder itself, and its length is called the \emph{width} of $C$.

The moduli space of compact translation surfaces of fixed genus $g$ and with the same number and order of singularities $\kappa_1,\dots,\kappa_s$ is called a \emph{stratum} and is denoted $\cH(\kappa_1,\dots,\kappa_s)$.
The genus is univocally determined by the well-known formula for zeros of  holomorphic 1-forms on a compact Riemann surface $\sum_i \kappa_i=2g-2$.

There is a natural action of the group $\GL^+(2,\RR)$ on translation surfaces, given by post-composition with the local charts.
The action of an element $g\in\GL^+(2,\RR)$ on $(M,\omega)$ will be denoted $g\cdot(M,\omega)$.
In particular, we will be interested in the \emph{Teichm\"uller geodesic flow}, that is the action of the group of diagonal matrices $G_t=\operatorname{diag}(e^t,e^{-t})$, for $t\in\RR$, and the one given by rotations
\[
	r_\theta = \rot{\theta}.
\]
Since the action of $\GL^+(2,\RR)$ preserves the topological structure of the surface it can be restricted to an action of each stratum $\cH(\kappa_1,\dots,\kappa_s)$.

The group $\operatorname{Aff}^+(M,\omega)$ of \emph{affine automorphisms} of a translation surface is the group of all orientation preserving homeomorphisms that map singular points to singular points, are diffeomorphisms on $M\setminus\Sigma$ and are affine on the same set with respect to the coordinates given by locally integrating $\omega$.
Under the identification between tangent planes at points $p\in M\setminus\Sigma$ and $\RR^2$ given by the local coordinates, the derivative of any affine automorphism coincides with its linear part and is a constant $2 \times 2$ real matrix.
Any affine diffeomorphism preserves the area of the surface and hence its derivative has determinant $1$.
We thus have a well defined map $D\colon\operatorname{Aff}^+(M,\omega)\to\SL(2,\RR)$.
The image of this map is called the \emph{Veech group} of $M$.
The kernel of this map is the group of \emph{translation equivalences}, that is affine automorphisms whose derivative is $\id$.
Two translation surfaces are said to be \emph{translation equivalent} if there is an affine diffeomorphism between the two whose derivative is the identity matrix.

Given a translation surface $(M,\omega)$ a \emph{translation cover} $(\widetilde{M},\widetilde{\omega})$ of $(M,\omega)$ is a cover $p\colon\widetilde{M}\to M$ such that $\widetilde{M}$ is a translation surface, $\widetilde{\omega}=p^*(\omega)$ and the covering map $p$ is locally given by translations in $\widetilde{M}\setminus p^{-1}(\Sigma)$.
Since $(\widetilde{M},\widetilde{\omega})$ is a translation surface, on it we can define the straight-line flow $\tflow$ in direction $\theta$.

Following~\cite{HooperWeiss}, one can give a more concrete definition of a translation cover in the case when the covering group is $\ZZ^2$, in other words, the surface $\widetilde{M}/\ZZ^2$ is homeomorphic to $M$.
In this case all $\ZZ^2$-covers of a compact connected translation surface $(M, \omega)$ are in one-to-one correspondence, up to isomorphism, with linearly independent pairs of absolutely homology classes $(\gamma_1, \gamma_2) \in H_1 (M; \ZZ)^2$.
We write $\Gamma = (\gamma_1,\gamma_2)$ for such a pair and denote the covering surface with $(\widetilde{M}_\Gamma,\widetilde{\omega})$. %where $M=\widetilde{M}_\Gamma/\ZZ^2$ and $\widetilde{\omega}$ is the pullback of $\omega$ under the covering map
If we denote the algebraic intersection form on $M$ with $\langle\blank, \blank\rangle \colon H_1(M;\ZZ) \times H_1(M;\ZZ) \to\ZZ$, then the lift of a closed curve $\sigma$ on the surface $M$ is a path $\widetilde{\sigma} \colon [t_0,t_1] \to \widetilde{M}_\Gamma$ such that $\widetilde{\sigma}(t_1)=(n_1,n_2)\cdot\widetilde{\sigma}(t_0)$, where $(n_1,n_2)=\bigl( \langle\gamma_1,[\sigma]\rangle, \langle\gamma_2,[\sigma]\rangle \bigr)\in\ZZ^2$ and $\cdot$ is the $\ZZ^2$-action by deck transformations on $(\widetilde{M}_\Gamma, \omega_\Gamma)$.

A necessary condition, see~\cite{AvilaHubert:recurrence}, for recurrence of the flow $\tflow$ is the \emph{no-drift condition}, that is
\[
	\hol(\gamma_i)=\int_{\gamma_i} \omega= 0, \qquad \text{ for } i=1,2.
\]
For $\ZZ$-covers, as P.~Hooper and B.~Weiss showed in~\cite{HooperWeiss}, under the no-drift condition recurrence of $\tflow$ is a consequence of general principles: ergodicity of the flow $\flow$ on $M$ implies recurrence of $\tflow$ on $\widetilde{M}_\gamma$.
However, for $\ZZ^2$-covers this is not true, as shown by V.~Delecroix in~\cite{Delecroix:divergent}.
In the following, we will always assume that the no-drift condition is satisfied.
The group of homology classes having zero holonomy will be denoted $H_1^{(0)}(M;\ZZ)$.

\subsection{Cocycles and essential values}
We now recall some definitions from infinite ergodic theory that we are going to need in order to state our criterion for ergodicity of $\tflow$ on a $\ZZ^2$-cover of a compact translation surface.
For more details on the subject we refer to~\parencite[][\textsection8.1-8.2]{Aaronson:infinite} or~\cite{Schmidt:cocycles}.

Let $F_t\colon X\to X$ be a flow on the measure space $(X,\mu)$, where $\mu$ is a non-atomic probability measure which is invariant under the action of the flow.
A measurable \emph{cocycle} $\alpha\colon X\times\RR\to\ZZ^2$ is a function satisfying
\[
%\begin{equation}\label{eq:defcocycle}
	\alpha(x,t+s)=\alpha(x,s)+\alpha(F_tx,s),
%\end{equation}
\]
for every $x\in X$ and $t$ and $s$ in $\RR$.
With these two ingredients we define a flow $\widetilde{F}_t$ on the space $\widetilde{X}=X\times\ZZ^2$, equipped with the natural measure, setting
\[
	\widetilde{F}_t(x,n)=(F_tx,n+\alpha(F_tx,n)).
\]
This flow is called a \emph{$\ZZ^2$-valued skew-product} of $F_t$.
An element $(n_1,n_2)\in\ZZ^2$ is called an \emph{essential value} for the skew-product if, for any measurable set $A\subset X$ of positive measure, there is a set of times $t$ with positive measure such that
\[
	\mu\bigl(\Set{x\in A : F_tx\in A, \alpha(x,t)=(n_1,n_2)}\bigr)>0.
\]
One can show that the set of essential values is a subgroup of $\ZZ^2$.
We will use the following Theorem, due to K.~Schmidt~\parencite[][Corollary~5.4]{Schmidt:cocycles}.

\begin{thm}\label{thm:Schmidt}
Let $F_t$ be an ergodic flow on a non-atomic measure space $(X,\mu)$ and let $\alpha\colon X\times\RR\to\ZZ^2$ be a cocycle.
Then the skew-product $\widetilde{F}_t$ is ergodic if and only if the set of essential values coincides with $\ZZ^2$.
\end{thm}

Given a $\ZZ^2$-cover $\widetilde{M}_\Gamma$ of a compact translation surface $M$ determined by two linearly independent curves $\gamma_1$ and $\gamma_2$ in $H_1^{(0)}(M;\ZZ)$, we can realise the directional flow $\widetilde{\phi}_t^\theta$ as a skew-product of  the flow $\phi_t^\theta$ on $M$ in the following way.
Choose an arbitrary point $\bar{x}\in M$ and, for every other point $x$ choose a continuous path $\gamma_{x,\bar{x}}$ from $x$ to $\bar{x}$.
Moreover, write $\gamma_{\bar{x},x}$ for the path that has the same image and opposite orientation.
Then define the $\ZZ^2$-valued cocycle $\alpha$ by
\[
	\alpha(x,t)=\bigl(\langle\gamma_1,\eta_{x,t}\rangle,\langle\gamma_2,\eta_{x,t}\rangle\bigr),
\]
where $\eta_{x,t}$ is the (homology class of the) closed path that connects $\bar{x}$ to $x$ along $\gamma_{\bar{x},x}$, then flows $x$ in direction $\theta$ for time $t$ up to $\phi_t^\theta(x)$ and finally closes up along $\gamma_{\phi_t^\theta(x),\bar{x}}$.
One can easily verify that $\alpha$ is indeed a cocycle and that the skew-product of $\phi_t^\theta$ over $\alpha$ is measurably equivalent to the flow $\widetilde{\phi}_t^\theta$.
Different choices in the definition of $\alpha$ lead to different skew-products, which are all measurably isomorphic to each other.

\section{Ergodicity Criterion}\label{sec:ergodicitycriterion}
In this section we prove a criterion for ergodicity of the directional flow $\widetilde{\phi}_t^\theta$ on a $\ZZ^2$-cover of a compact translation surface.
Our criterion is a generalisation of the one proven by P.~Hubert and B.~Weiss in~\cite{HubertWeiss}.

If $C$ is a cylinder in the compact surface $M$, we denote with $\delta(C)\in H_1(M;\ZZ)$ the homology class of a core curve of $C$.
We write $k(C)=(\langle\gamma_1,\delta(C)\rangle, \langle\gamma_2,\delta(C)\rangle)\in\ZZ^2$ and $v(C)$ for the displacement vector of $\delta(C)$.
Finally, let $A(C)$ be the area of the cylinder $C$.
From the description of the $\ZZ^2$-cover we gave above it follows that, if $k(C)\neq(0,0)$, then the lift $\widetilde{C}$ of $C$ to $\widetilde{M}$ is an infinite strip.
We recall the following Definition, first introduced in~\cite{HubertWeiss}.

\begin{definition}\label{def:wellapproximated}
A direction $\theta\in S^1$ is \emph{well approximated by strips} if there are $\epsilon>0$, $k_\theta\neq(0,0)$ and infinitely many strips $\widetilde{C}\subset\widetilde{M}$ such that $k(C)\equiv k_\theta$, $A(C)>\epsilon$ and
\[
	\abs{(\cos\theta,\sin\theta)\wedge v(C)} \leq (1-\epsilon)\frac{A(C)}{2\norm{v(C)}}.
\]
\end{definition}

Well approximated directions are related with essential values.
More precisely one has the following

\begin{prop}\label{thm:stripsessentialvalues}
Suppose $\theta\in S^1$ is a direction that is well approximated by strips.
Then $k_\theta$ is an \emph{essential value} for the straight-line flow $\tflow$ on $\widetilde{M}$.
\end{prop}

The Proposition is proved in~\cite{HubertWeiss} for the case of $\ZZ$-covers, see their Claim~12.
The proof holds verbatim for $\ZZ^2$-covers, with the obvious modifications.
In particular, they construct embedded rectangles on the $\ZZ$-cover with sides in direction $\theta$ and $\theta+\frac{\pi}{2}$ and opposite corners at the points $x$ and $k\cdot x$, where $k\in\ZZ$, analogously to $k_\theta$ in Definition~\ref{def:wellapproximated}, is the value of the algebraic intersection form between a core curve of a cylinder and the curve giving the $\ZZ$-cover, and $\cdot$ represents the $\ZZ$-action on the cover itself.
We modify that definition asking the corners of rectangles to be placed at $x$ and $k_\theta\cdot x$, where $\cdot$ now represents the $\ZZ^2$ action on $\widetilde{M}_\Gamma$.

We will sometimes say that a sequence of strips $\widetilde{C_n}$ well approximating a direction $\theta\in S^1$ \emph{produces} the essential value $k_\theta$.
We can finally state our ergodicity criterion.

\begin{prop}[Ergodicity Criterion]\label{thm:criterion}
Let $\theta\in S^1$ be an ergodic direction for the directional flow on the compact translation surface $(M,\omega)$.
If $\theta$ is well approximated by two sequences of strips with $k_\theta^h\equiv(\pm1,0)$ and $k_\theta^v\equiv(0,\pm1)$, then the flow $\widetilde{\phi}_t^\theta$ on the surface $(\widetilde{M}_\Gamma,\widetilde{\omega})$ is \emph{ergodic}.
\end{prop}

\begin{proof}
Let $\widetilde{C_n^h}$ and $\widetilde{C_n^v}$ be two sequences of ``horizontal'' and ``vertical'' strips that well approximate $\theta$.
If $p\colon\widetilde{M}\to M$ is the covering map, write $C_n^h=p\Bigl(\widetilde{C_n^h}\Bigr)$ and $C_n^v=p\Bigl(\widetilde{C_n^v}\Bigr)$.
Suppose we have $k(C_n^h)=k_\theta^h\equiv(\pm1,0)$ and $k(C_n^v)=k_\theta^v\equiv(0,\pm1)$.
Then, by Proposition~\ref{thm:stripsessentialvalues}, $k_\theta^h$ and $k_\theta^v$ are essential values for the skew-product $\widetilde{\phi}_t^\theta$.
Since, as we recalled earlier, the set of essential values is a closed subgroup of $\ZZ^2$ and we have shown that two generators of this group are essential values, the conclusion now follows directly from Theorem~\ref{thm:Schmidt}.
\end{proof}

To show that the translation flow $\flow$ on the compact surface $(M,\omega)$ is ergodic, as in the hypothesis of our ergodicity criterion, we will use a classical result, due to H.~Masur~\cite{Masur:criterion}.

\begin{thm}[Masur's criterion~\cite{Masur:criterion}]\label{thm:Masurcriterion}
Let $(M,\omega)\in\cH(\kappa_1,\dots,\kappa_s)$ be a compact translation surface.
Let $g\in\SL(2,\RR)$ a matrix that sends the direction $\theta$ on the vertical direction.
Suppose that there exist a bounded subset $B\subset\cH(\kappa_1,\dots,\kappa_s)$ and a sequence of times $t_n\to+\infty$ such that $G_{t_n}(g\cdot(M,\omega))\in B$ for all $n\in\NN$.
Then the directional flow $\flow$ on $(M,\omega)$ is \emph{uniquely ergodic}.
\end{thm}

It is worth to stress that the core of our proof is in showing that we can construct the sequence of strips as in the statement of Proposition~\ref{thm:criterion}, that is the content of the following three sections.

\section{Construction of Ergodic Directions}\label{sec:ergodicdirections}
%We will now explain how one can construct two families of strips satisfying the hypothesis of Proposition~\ref{thm:criterion} on a surface $ (\widetilde{M}_\Gamma,\widetilde{\omega})$.
We define the punctured torus
\[
	\TT_0^2=\left( \RR^2\setminus\Set{\left( \frac{n}{2}, \frac{m}{2} \right), n,m\in\ZZ} \right)/ \ZZ^2.
\]
A fundamental domain for it is the subset
\begin{equation}\label{eq:fundamentaldomain}
	T_0^2 = \left[ -\tfrac{1}{2}, \tfrac{1}{2} \right) \times \left[ -\tfrac{1}{2}, \tfrac{1}{2} \right) \setminus \left\{ (0,0), \left(-\tfrac{1}{2},-\tfrac{1}{2}\right), \left(-\tfrac{1}{2},0\right), \left(0,-\tfrac{1}{2}\right)\right\}\subset\RR^2.
\end{equation}
Finally, let $\pi\colon\TT_0^2\to T_0^2$ be the bijection, induced by the quotient map, between the two.
In the following, unless explicitly stated, we will identify $\TT_0^2$ with its fundamental domain $T_0^2$ via $\pi$.

For $z\in\TT_0^2$, let $M(z)\in\cH(1,1)$ the surface represented in Figure~\ref{fig:mz}, made out of two square tori glued along a slit.
We will distinguish the two singularities of $M(z)$ one from the other.
In particular $z=(x,y)$ will always be the position of the singularity denoted by $\bullet$ and $-z$ the one of the singularity denoted by $\circ$.
Let $\omega$ be the $1$-form induced by $dz$ on $M(z)$, and call $\cL\subset\cH(1,1)$ the locus made of all the surfaces $(M(z),\omega)$ for $z\in \TT_0^2$.
Remark that the precise choice of the slit joining the two singularities in the fundamental domain does not affect the flat geometry of $M(z)$, as two different choices are translation equivalent to each other.
However, this is relevant for the analysis of the homology of the surface we are going to carry.
To this end, we will always choose the slit obtained by projecting on the tori the straight segment in the plane joining the points $\pi(z)$ and $-\pi(z)$, as it is shown in Figure~\ref{fig:mz}.

The linear action of $\GL(2,\RR)$ on the plane $\RR^2\cong\CC$ induces an action of the subgroup $\GL(2,\ZZ)$ on the torus $\TT_0^2$.
One can show, see~\parencite[][p.~648 and p.~652--654]{FraczekUlcigrai:tube}, that the locus $\cL$ is preserved by this action.
In fact, the surface $g\cdot M(z)$ is translation equivalent to the surface $M(gz)$, for $g\in\GL(2,\ZZ)$.

\begin{figure}[t]
\centering
\def\svgwidth{0.6\textwidth}
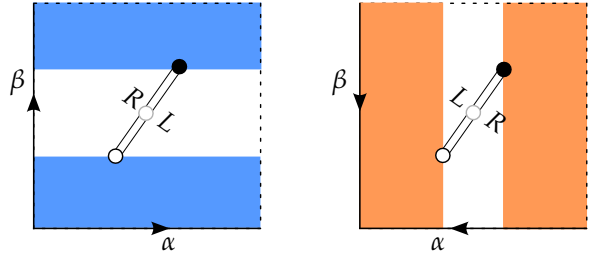
\caption{The surface $M(z)$ with the cylinder $C_z^h$ in light blue and the cylinder $C_z^v$ in orange.}
\label{fig:mz}
\end{figure}

The group of translation equivalences of $M(z)$ consists of two elements, the identity $\id$ and the involution $\tau$ that exchanges the two squares by translations.
One has
\[
	H_1^{(0)}(M(z);\ZZ) =\Set{\gamma\in H_1(M(z);\ZZ) : \tau_*\gamma=-\gamma},
\]
where $\tau_*$ denotes the induced action on the homology.
For every surface $M(z)\in\cL$ let $\{\alpha,\beta\}$ be the basis of $H_1^{(0)}(M(z);\ZZ)$ as in Figure~\ref{fig:mz}.
To any such surface we associate a $\ZZ^2$-cover $(\widetilde{M}(z)_\Gamma,\widetilde{\omega})$, where $\Gamma=(\beta,-\alpha)$ and $\widetilde{\omega}$ is the pullback under the covering map of $\omega$.
The negative sign in $\alpha$ is due to the orientation chosen for $\alpha$, see Figure~\ref{fig:mz}.
We remark that the orientation covering $\widetilde{M}(\ZZ^2,R)$ of $F(\ZZ^2,R)$ coincides with $\widetilde{M}(R)_\Gamma$.

Representing every surface in $\cL$ as $M(z)$, for some $z\in \TT_0^2$, allows us to consistently choose a standard basis of $H_1^{(0)} (M(z);\ZZ)$, given by the curves $\alpha$ and $\beta$ as in Figure~\ref{fig:mz}.
Remark that this is equivalent to choose a \emph{marking} of the translation surfaces considered.
Under this choice of basis we can represent the induced action of a matrix $g\in\GL(2,\ZZ)$ on the zero holonomy homology
\[
	g_*(z)\colon H_1^{(0)}(M(z);\ZZ)\to H_1^{(0)}(g\cdot M(z);\ZZ)
\]
with a $2\times 2$ integer matrix.
More precisely, as the group of translation equivalences of $M(z)$ consists of two elements, there are exactly two maps $\zeta^g,\tau\circ\zeta^g\colon M(z)\to g\cdot M(z)$ that have a fixed matrix $g$ as derivative.
Their induced action from $H_1^{(0)}(M(z);\ZZ)$ to $H_1^{(0)}(g\cdot M(z);\ZZ)$ is related by $(\tau\circ\zeta^g)_*=-\zeta_*^g$.
In other words, given a matrix $g\in\GL(2,\ZZ)$, its induced action on the zero holonomy homology $g_*(z)$ is well defined only as an element of $\PGL(2,\ZZ)$.
One has
\begin{equation}\label{eq:chainrule}
	(g_1\cdot g_2)_*(z)=(g_1)_*(g_2z)\cdot(g_2)_*(z).
\end{equation}

\begin{convention}
Since we are interested only in the action induced on the zero holonomy homology, for the sake of brevity we will write $g_*(z)$ to denote the action from $H_1^{(0)}(M(z);\ZZ)\subset H_1(M(z);\ZZ)$ to $H_1^{(0)}(g\cdot M(z);\ZZ)$.
Moreover, the notation $g_*(z)$ is meant to stress the fact that the induced action of the matrix $g$ from $H_1^{(0)}(M(z);\ZZ)$ to $H_1^{(0)} (g\cdot M(z);\ZZ)$ depends on $z$, the endpoint of the slit, see Lemma~\ref{thm:keylemma} for more details.
\end{convention}

\subsection{Construction of the strips}
Fix $z\in \TT_0^2$ and let $\pi(z)=(x,y)\in T_0^2$.
We will consider the following two cylinders in $M(z)$.
\begin{gather}
	C_z^h=\left[-\tfrac{1}{2},\tfrac{1}{2}\right) \times \left( \left[\abs{y},\tfrac{1}{2}\right) \cup \left[-\tfrac{1}{2},-\abs{y}\right]\right),\label{eq:horizontalcylinders}\\
	C_z^v=\left( \left[\abs{x},\tfrac{1}{2}\right) \cup \left[-\tfrac{1}{2},-\abs{x}\right]\right) \times \left[-\tfrac{1}{2},\tfrac{1}{2}\right)\label{eq:verticalcylinders},
\end{gather}
which are represented in Figure~\ref{fig:mz}.
As the notation suggests, we will use the former to obtain a family of strips well approximating a direction $\theta$ and with $k_\theta^h\equiv(\pm1,0)$; the latter will lead to a family of strips with $k_\theta^v=(0,\pm1)$.
We will sometimes call also the image under a matrix $g\in\GL(2,\ZZ)$ of the first cylinder ``horizontal'' and of the second one ``vertical''.

The displacement vectors of the core curves of these two cylinders are $v(C_z^h)=(1,0)$ and $v(C_z^v)=(0,1)$.
We have $k(C_z^h)=(\langle\delta(C_z^h),\beta\rangle, \langle\delta(C_z^h),\alpha\rangle)=(1,0)$ and $k(C_z^v)=(0,1)$.
Finally, their areas are given by $A(C_z^h)=1-2\abs{y}$ and $A(C_z^v)=1-2\abs{x}$.

Let $g\in\SL(2,\ZZ)$.
We will denote $z_g=(x_g,y_g)=g^{-1}(z)$ and $\zeta^g\colon M(z_g)\to M(z)$ one of the affine transformations that has $g$ as derivative.
Write $C_g^h=\zeta^g(C_{z_g}^h)\subset M(z)$.
Its core curve is $\delta(C_g^h)=\zeta_*^g(\delta(C_{z_g}^h))$.
We have
\begin{gather*}
	v(C_g^h)=v(\zeta_*^g(\delta(C_{z_g}^h)))=\hol(\zeta_*^g(\delta(C_{z_g}^h)))=D\zeta^g\hol(\delta(C_{z_g}^h))=g(1,0),\\
	A(C_g^h)=A(\zeta^g(C_{z_g}^h))=A(C_{z_g}^h)=1-2\abs{y_g}.
\end{gather*}
Moreover
\[
\begin{split}
	k(C_g^h)&=(\langle\zeta_*^g(\delta(C_{z_g}^h)),\beta\rangle,\langle\zeta_*^g(\delta(C_{z_g}^h)),\alpha\rangle)\\
		&= (\langle\delta(C_{z_g}^h),(\zeta_*^g)^{-1}\beta\rangle,\langle\delta(C_{z_g}^h),(\zeta_*^g)^{-1}\alpha\rangle).
\end{split}
\]
If we assume that $g_*(z_g)\colon H_1^{(0)}(M(z_g);\ZZ) \to H_1^{(0)}(M(z);\ZZ)=\id$ we have
\[
	k(C_g^h)=\pm (\langle\delta(C_{z_g}^h),\beta\rangle,\langle\delta(C_{z_g}^h),\alpha\rangle)=(\pm1,0).
\]
Similar computations hold also for the ``vertical'' cylinders $C_g^v$.
%We will now show that, under the assumption that $g_*(g^{-1}(z))\equiv\id$ on $H_1^{(0)}(M(z_g);\ZZ)$, the two sequences of strips obtained by lifting $C_{z_g}^h$ and $C_{z_g}^v$ well approximate $\theta$.

In order to state our result we need to introduce some notation.
Given a real number $0<x<1$ its continued fraction expansion is denoted
\[
	x=[0;a_1,a_2,\dots]=\cfrac{1}{a_1+\cfrac{1}{a_2+\cdots}},
\]
moreover call
\[
	h_+=\begin{pmatrix}
			1	&	1\\
			0	&	1
		\end{pmatrix}
	\qquad \text{and} \qquad
	h_-=\begin{pmatrix}
			1	&	0\\
			1	&	1
		\end{pmatrix}.
\]

\begin{thm}\label{thm:conditionaltheorem}
Suppose that $z=(x,y)\in \TT_0^2$, let $\vartheta=[0;a_1,a_2,\dots]$ and fix $\epsilon>0$.
Assume that there is a sequence of even times $k_n$ such that
\[
	\Bigl( h_+^{a_1}\cdots h_-^{a_{k_n}}\bigr)\cdot M(z) = M(z) \qquad \text{ and } \qquad	\Bigl( h_+^{a_1}\cdots h_-^{a_{k_n}} \Bigr)_*(z)=\id.
\]
Finally, suppose
\[
	a_{k_n}\geq\frac{4(1+\epsilon)}{1-2\abs{x}}
	\qquad \text{and} \qquad
	a_{k_n+1}\geq\frac{2(1+\epsilon)}{1-2\abs{y}}.
\]
Then the directional flow $\tflow$ in direction $\theta$ of the vector $(1,\vartheta)$ on the $\ZZ^2$-cover $\widetilde{M}(z)_\Gamma$, where $\Gamma=(\beta,-\alpha)$, is \emph{ergodic}.
\end{thm}

\begin{proof}
We divide the proof in three steps.

\textit{Step 1: the flow $\flow$ on $M(z)$ is \emph{ergodic}.}
Denote by $\tfrac{p_n}{q_n}$ the $n$-th convergent of the continued fraction expansion of $\vartheta$.
Then
\[
	h_+^{a_1}\cdots h_-^{a_{k_n}}=
		\begin{pmatrix}
			q_{k_n}	&	q_{k_n-1}\\
			p_{k_n}	&	p_{k_n-1}
		\end{pmatrix}.
\]
The matrix
\[
	\sigma=
		\begin{pmatrix}
			\vartheta	&	-1\\
			0	&	\frac{1}{\vartheta}
		\end{pmatrix}
\]
sends the vector $(1,\vartheta)$ to the vertical direction.
Let us show that the sequence
\[
	(G_{\log{k_n}}\cdot\sigma\cdot M(z))_{n\in\NN} \subset \cH(1,1)
\]
is bounded in the stratum.
We have
\[
	G_{\log{k_n}}\cdot\sigma\cdot M(z) = \sigma_n\cdot M(z),
\]
where
\[
\begin{split}
	\sigma_n 	&=	\operatorname{diag}\Bigl(q_{k_n},\tfrac{1}{q_{k_n}}\Bigr)\cdot\sigma\cdot h_+^{a_1}\cdots h_-^{a_{k_n}}\\
			&=	\begin{pmatrix}
					q_{k_n}	&	0\\
					0		&	\frac{1}{q_{k_n}}
				\end{pmatrix}
				\begin{pmatrix}
					\vartheta		&	-1\\
					0		&	\frac{1}{\vartheta}
				\end{pmatrix}
				\begin{pmatrix}
					q_{k_n}	&	q_{k_n-1}\\
					p_{k_n}	&	p_{k_n-1}
				\end{pmatrix}\\
			&=	\begin{pmatrix}
					q_{k_n}(q_{k_n}\vartheta-p_{k_n})	&	q_{k_n}(q_{k_n-1}\vartheta-p_{k_n-1}) \\
					\frac{p_{k_n}}{\vartheta q_{k_n}}	&	\frac{p_{k_n-1}}{\vartheta q_{k_n}}
				\end{pmatrix}.
\end{split}
\]
All the entries of $\sigma_n$ are contained in $[-1,1]$.
In fact
\[
	\abs{q_{k_n}(q_{k_n}\vartheta-p_{k_n})}<\abs{q_{k_n}(q_{k_n-1}\vartheta-p_{k_n-1})}<1
\]
and, since we assumed that $k_n$ are even numbers,
\[
	0<\frac{p_{k_n-1}}{q_{k_n}}<\frac{p_{k_n}}{q_{k_n}}<\vartheta.
\]
Call $G_0\subset\SL(2,\RR)$ the compact subset of matrices with all coefficients in $[-1,1]$.
Then the orbit $G_0\cdot M(z)$ is a compact subset of $\cH(1,1)$.
Since we have shown that $\sigma_n\cdot M(z)\in G_0 \cdot M(z)$ for all $n$, we can apply Masur's criterion (Theorem~\ref{thm:Masurcriterion}) and deduce that the flow in the direction $\theta$ of the vector $(1,\vartheta)$ is uniquely ergodic on the compact surface $M(z)$.

\textit{Step 2: $\theta$ is well approximated by ``horizontal'' cylinders.}
The hypothesis guarantee, for every $n$, the existence of transformations
\[
	\zeta_n\colon M(z)\to M(z)
\]
whose derivative is $D\zeta_n=h_+^{a_1}\dots h_-^{a_{k_n}}$ and that act trivially on the holonomy zero homology.
Consider the cylinders $C_n^h=\zeta_n(C_z^h)\subset M(z)$, where $C_z^h$ was defined in~\eqref{eq:horizontalcylinders}.
By the assumptions, and by the computations before the statement of the Theorem, we have
\begin{gather*}
	k(C_n^h)=(1,0),\\
	A(C_n^h)=A(C_z^h)=1-2\abs{y},\\
	v(C_n^h)=(D\zeta_n) v(C_z^h)= (D\zeta_n)(1,0)=(q_{k_n},p_{k_n}).
\end{gather*}
Since $k(C_n^h)\neq(0,0)$ the cylinders lift to strips in the infinite surface $\widetilde{M}(z)_\Gamma$.
We now show that these strips well approximate the direction $\theta$.
On one hand we have
\[
	\frac{\abs{(1,\vartheta)\wedge v(C_n^h)}}{\norm{(1,\vartheta)}}	=\frac{\abs{(1,\vartheta)\wedge(q_{k_n},p_{k_n})}}{\norm{(1,\vartheta)}}
										=\frac{\abs{q_{k_n}\vartheta-p_{k_n}}}{\norm{(1,\vartheta)}}
										<\frac{1}{\norm{(1,\vartheta)}}\frac{1}{a_{k_n+1}q_{k_n}},
\]
on the other
\[
	\frac{A(C_n^h)}{2\norm{(C_n^h)}} \geq \frac{1-2\abs{y}}{2q_{k_n}\norm{(1,\vartheta)}}.
\]
So 
\[
	\frac{\abs{(1,\vartheta)\wedge v(C_n^h)}}{\norm{(1,\vartheta)}} \leq\frac{1}{1+\epsilon}\frac{A(C_n^h)}{2\norm{(C_n^h)}}
\]
if $a_{k_n+1}\geq\frac{2(1+\epsilon)}{1-2\abs{y}}$.

\textit{Step 3: $\theta$ is well approximated by ``vertical'' cylinders.}
We now turn our attention to the cylinders $C_n^v=\zeta_n(C_z^v)$, where $C_z^v$ was defined in~\eqref{eq:verticalcylinders} and $\zeta_n$ is defined in the previous step.
Let us recall that
\begin{gather*}
	k(C_n^v)=(0,1),\\
	A(C_n^v)=A(C_z^v)=1-2\abs{x},\\
	v(C_n^v)=(D\zeta_n)v(C_z^v)=(D\zeta_n)(0,1)=(q_{k_n-1},p_{k_n-1}).
\end{gather*}
As before, these cylinders all lift to infinite strips in $\widetilde{M}(z)_\Gamma$.
Remark that, since $\alpha<\frac{p_{k_n-1}}{q_{k_n-1}}$ the analysis we carried on in the previous step cannot be applied in this case and we have to proceed in a slightly different way.
We have
\[
	\frac{\abs{(1,\vartheta)\wedge v(C_n^v)}}{\norm{(1,\vartheta)}}	=\frac{\abs{(1,\vartheta)\wedge(q_{k_n-1},p_{k_n-1})}}{\norm{(1,\vartheta)}}
										%=\frac{\abs{q_{k_n-1}\vartheta-p_{k_n-1}}}{\norm{(1,\vartheta)}}
										<\frac{1}{\norm{(1,\vartheta)}}\frac{1}{a_{k_n}q_{k_n-1}},
\]
and
\[
	\frac{A(C_n^v)}{2\norm{v(C_n^v)}} = \frac{1-2\abs{x}}{2\norm{(q_{k_n-1},p_{k_n-1})}}.
\]
Imposing
\[
	\frac{1}{\norm{(1,\vartheta)}}\frac{1}{a_{k_n}q_{k_n-1}}\leq\frac{1}{1+\epsilon}\frac{A(C_n^v)}{2\norm{v(C_n^v)}},
\]
we obtain the inequality
\[
	a_{k_n}\geq\frac{\norm{(q_{k_n-1},p_{k_n-1})}}{q_{k_n-1}\norm{(1,\vartheta)}}\frac{2(1+\epsilon)}{1-2\abs{x}}.
\]
Let us focus on the first ratio in the RHS.
Using that $k_n$ is even and that $a_{k_n}q_{k_n-1}< q_{k_n}=a_{k_n}q_{k_n-1}+q_{k_n-2}<(a_{k_n}+1)q_{k_n-1}$ we have
\[
	\frac{\norm{(q_{k_n-1},p_{k_n-1})}}{q_{k_n-1}\norm{(1,\vartheta)}}	<\frac{q_{k_n}\norm{(q_{k_n-1},p_{k_n-1})}}{q_{k_n-1}\norm{(q_{k_n},p_{k_n})}}
								<\frac{a_{k_n}+1}{a_{k_n}}\frac{\norm{(q_{k_n-1},p_{k_n-1})}}{\norm{(q_{k_n-1},p_{k_n-1})}}
											\leq 2,
\]
which finally gives
\[
	a_{k_n}\geq\frac{4(1+\epsilon)}{1-2\abs{x}}\geq\frac{\norm{(q_{k_n-1},p_{k_n-1})}}{q_{k_n-1}\norm{(1,\vartheta)}}\frac{2(1+\epsilon)}{1-2\abs{x}}.
\]

The three steps we just completed show that all the hypothesis of the ergodicity criterion~\ref{thm:criterion} holds and hence we have shown that the flow $\tflow$ on $\widetilde{M}(z)_\Gamma$ is ergodic.
\end{proof}

\subsection{The action of $\SL(2,\RR)$ on the homology}
Thanks to Theorem~\ref{thm:conditionaltheorem} we have reduced ourselves to construct transformations in $\SL(2,\RR)$ that act trivially on the zero holonomy homology.
In order to construct such elements, we need to analyse in more detail the induced action of $\SL(2,\RR)$.
Call $\omega=\left(\begin{smallmatrix}	0&1 \\	-1&0 \end{smallmatrix}\right)$.
We have the following identities, which can be verified simply by matrix multiplications and that will be crucial in our constructions
\begin{equation}\label{eq:identities}
	\omega\cdot h_\pm\cdot\omega^{-1}=h_\mp^{-1},\qquad	h_-\cdot h_+^{-1}\cdot h_-=\omega^{-1},\qquad	h_+\cdot h_-^{-1}\cdot h_+=\omega.
\end{equation}

\begin{lemma}
Let $F=\Set{(x,y)\in T_0^2 : x,y\neq-\tfrac{1}{2}}$.
Then for every $z\in F$
\[
	(-\id)_*(z)=\id.
\]
Let $g$ be an element of $\SL(2,\ZZ)$ and let $g\cdot M(z)=M(z')$, with $z,z'\in F$.
Then
\begin{equation}\label{eq:gidentities}
	g\cdot M(-z)=M(-z') \qquad\text{and}\qquad g_*(-z)=g_*(z).
\end{equation}
\end{lemma}

\begin{proof}
The involution $-\id$ acts on $\cL$ by rotating each square of $M(z)$ by the angle $\pi$.
In particular, the induced action on $H_1^{(0)}(M(z);\ZZ)$ sends $\alpha$ to $-\alpha$ and $\beta$ to $-\beta$, unless $z$ belongs to the boundary of $T_0^2$.
The first claim follows exchanging the two squares that form $M(z)$ using the affine automorphism $\tau$.

The second one follows from the fact that $g\cdot M(z)=M(gz)$ and that $g$ and $-\id$ commute, so we have $g\cdot M(-\id z)=M(-\id z')$.
Finally, using~\eqref{eq:chainrule} we have
\[
\begin{split}
	g_*(z)&=(-\id)_*(z')\cdot g_*(z)= (-\id\cdot g)_*(z)\\
		&=(g\cdot-\id)_*(z)=g_*(-z)\cdot(-\id)_*(z)=g_*(-z).\qedhere
\end{split}
\]
\end{proof}

Since $h_+$ and $h_-$ generate $\SL(2,\ZZ)$, if we look at their action on the zero holonomy homology we obtain complete informations on the action of $\SL(2,\ZZ)$ itself.
The next result is proved in~\cite[][Lemma~3.3]{FraczekUlcigrai:tube}, to which we refer for the proof.

\begin{lemma}\label{thm:keylemma}
Set $S=\Set{(x,y)\in T_0^2 : -\tfrac{1}{2}\leq x+y<\tfrac{1}{2}}$, see Figure~\ref{fig:S}.
For every $z\in T_0^2$ we have
\[
	(h_\pm)_*(z)=
		\begin{cases}
			h_\pm,		&	\text{if $z\in S$,}\\
			h_\pm^{-1},		&	\text{if $z\notin S$.}
		\end{cases}
\]
\end{lemma}

\begin{figure}[t]
\centering
\def\svgwidth{0.5\textwidth}
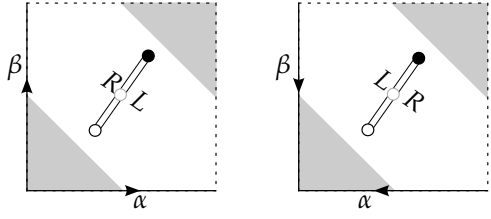
\caption{The set $S$ defined in the Lemma~\ref{thm:keylemma}.}
\label{fig:S}
\end{figure}

The previous Lemma enables us to evaluate precisely the action induced from $H_1^{(0)}(M(z);\ZZ)$ to $H_1^{(0)}(h_\pm \cdot M(z);\ZZ)$ by $h_\pm$.
For a moment, it is worth to distinguish between the punctured torus $\TT_0^2$ and its fundamental domain $T_0^2$ defined by~\eqref{eq:fundamentaldomain}.
In fact, remark that
\[
	z\in S \iff h_\pm(z)\in T_0^2.
\]
Writing $\chi=\chi_{T_0^2}$ for the indicator function of $T_0^2$, we can thus rephrase Lemma~\ref{thm:keylemma} as
%\begin{equation}\label{eq:haction}
\[
	(h_\pm)_*(z)=h_\pm^{1-2\chi(h_\pm(z))}.
\]
%\end{equation}
In other words, if after applying the transformation $h_\pm$ the endpoint $z$ of the slit is outside the fundamental domain $T_0^2$, then the induced action on the zero holonomy homology is reduced by $2$.
Iterating this reasoning, we get the following

\begin{cor}\label{thm:homologyaction}
Suppose that we apply the transformation $h_-^n$, for some natural number $n$, to the surface $M(z)$, with $z\in \TT_0^2$.
%Then $h_-^n(z)=(x,y+nx)$.
Write 
\[
	y+nx=k+\{\!\{y+nx\}\!\},
\]
where $\{\!\{y+nx\}\!\}\in\bigl[-\tfrac{1}{2},\tfrac{1}{2}\bigr)$ is the displacement from the nearest integer and $k\in\ZZ$.
Then
\[
	(h_-^n)_*(z)=h_-^{n-2 \abs{k}}.
\]

Similarly, consider the transformation $h_+^m$, for some natural number $m$.
If $x+my=l+\{\!\{x+my\}\!\}$, where $l\in\ZZ$, we have $(h_+)_*(z)=h_+^{m-2\abs{l}}$.
\end{cor}

\section{Ergodic directions for rational $r$, $s$}
In this section we construct ergodic directions for the surface $M(z)_\Gamma$ under the assumption that $z$ is rational, that is $z=(r/2q, s/2q)\in \TT_0^2$ with $r,s,q\in\ZZ$, $\abs{r},\abs{s}<q$, and $s$ is non-zero and coprime with $q$.
We have

\begin{lemma}\label{thm:lemmaodd}
Suppose that at least one number $s$ or $r$ is odd.
Let $a, d$ be natural numbers such that
\begin{equation}\label{eq:conditionodd}
	4q<a,d \leq 6q \quad \text{and} \quad r+as \equiv -q \pmod{2q}, \quad ds-r\equiv -q\pmod{2q},
\end{equation}
and let $n=8qm$, for some $m\in\NN$.
Then, setting
\[
	g_z (n) = h_+^{d-1}\cdot h_-\cdot h_+\cdot h_-^d\cdot h_+^n\cdot h_-^{a-1}\cdot h_+\cdot h_-\cdot h_+^a\cdot h_-^n\in \SL(2,\ZZ),
\]
we have 
\[
	g_z(n) \cdot M(z) = M(z), \qquad \text{ and } \qquad (g_z(n))_*(z)=\id.
\]
\end{lemma}

\begin{proof}
Let us first show that if suffices to show the result when $s>0$.
In fact, if $s<0$ then we can consider $-z\in \TT_0^2$ and find a transformation such that $g_{-z}\cdot M(-z)=M(-z)$ and $(g_{-z})_*(-z)=\id$.
By assumption $z\in F$, so from~\eqref{eq:gidentities} we have $g_{-z}\cdot M(z)=M(z)$ and $(g_{-z})_*(z)=\id$, which is want we wanted.

Now write
\[
	(h_-^n)_*(z) = h_-^t.
\]
Remark that we have
\[
	t=n-2 \abs{4mr},
\]
since $h_-^n(z)=\bigl(\tfrac{r}{2q},\tfrac{s}{2q}+4mr\bigr)$.
Since $s$ and $q$ are coprime and either $s$ or $r$ is odd, there exist $a,d,k,l\in\NN$ such that $4q<a,d\leq 6q$ and
\begin{equation}\label{eq:conditionad}
	\begin{split}
		r+as &= 2qk-q, \\ 
		ds-r  &=2ql -q.
	\end{split}
\end{equation}
We  have
\begin{gather*}
	h_-^n\left(\frac{r}{2q},\frac{s}{2q}\right)=\left(\frac{r}{2q},\frac{s+8mr}{2q}\right),\\
	h_+^a\left(\frac{r}{2q},\frac{s}{2q}\right) = \left(\frac{r+as}{2q},\frac{s}{2q}\right) = \left(k-\frac{1}{2},\frac{s}{2q}\right),\\
	h_-\left(-\frac{1}{2},\frac{s}{2q}\right) = \left(-\frac{1}{2},\frac{s}{2q}-\frac{1}{2}\right), \qquad
	h_+\left(-\frac{1}{2},\frac{s}{2q}-\frac{1}{2}\right) = \left(\frac{s}{2q},\frac{s}{2q}-\frac{1}{2}\right),\\
	h_-^{a-1}\left(\frac{s}{2q},\frac{s}{2q}-\frac{1}{2}\right) = \left(\frac{s}{2q},\frac{(a-1)s+s-q}{2q}\right) = \left(\frac{s}{2q},\frac{-r+2(k-1)q}{2q}\right),\\
	h_+^n\left(\frac{s}{2q},-\frac{r}{2q}\right) = \left(\frac{s-8mqr}{2q},-\frac{r}{2q}\right) = \left(\frac{s}{2q}-r,-\frac{r}{2q}\right),\\
	h_-^d\left(\frac{s}{2q},-\frac{r}{2q}\right) = \left(\frac{s}{2q},\frac{ds-r}{2q}\right) = \left(\frac{s}{2q},l-\frac{1}{2}\right),\\
	h_+\left(\frac{s}{2q},-\frac{1}{2}\right) = \left(\frac{s}{2q}-\frac{1}{2},-\frac{1}{2}\right), \qquad
	h_-\left(\frac{s}{2q}-\frac{1}{2},-\frac{1}{2}\right) = \left(\frac{s}{2q}-\frac{1}{2}, \frac{s}{2q}\right),\\
	h_+^{d-1}\left(\frac{s}{2q}-\frac{1}{2}, \frac{s}{2q}\right) = \left(\frac{(d-1)s+s-q}{2q}, \frac{s}{2q}\right) = \left(\frac{2(l-1)q +r}{2q}, \frac{s}{2q}\right),
\end{gather*}
so $g_z(n)\cdot M(z)=M(z)$.

We now have to analyse the action of the induced transformation on the zero holonomy homology.
We write
\[
	(h_+^a)_*\left(\frac{r}{2q},\frac{s}{2q}\right) = h_+^{\tilde{a}}.
\]
Corollary~\ref{thm:homologyaction} gives us that $\tilde{a} = a -2k$, $k$ being the natural number defined by~\eqref{eq:conditionad}.
Since $a>4q$, we have that $\tilde{a}>0$.
Writing
\[
	(h_-^d)_*\left(\frac{s}{2q},-\frac{r}{2q}\right) = h_-^{\tilde{d}},
\]
with $\tilde{d}=d-2l>0$.
We want to show that we have
\[
\begin{split}
	(g_z(n))_*(z) &= h_+^{\tilde{d}+1}\cdot h_-^{-1}\cdot h_+\cdot h_-^{\tilde{d}}\cdot h_+^t\cdot h_-^{\tilde{a}+1}\cdot h_+^{-1}\cdot h_-\cdot h_+^{\tilde{a}}\cdot h_-^t\\
			&= h_+^{\tilde{d}}\cdot \omega\cdot h_-^{\tilde{d}}\cdot h_+^t\cdot h_-^{\tilde{a}}\cdot \omega^{-1}\cdot h_+^{\tilde{a}}\cdot h_-^t\\
			&= h_+^{\tilde{d}}\cdot h_+^{-\tilde{d}}\cdot h_-^{-t}\cdot h_+^{-\tilde{a}}\cdot h_+^{\tilde{a}}\cdot h_-^t\\
			&= h_-^{-t}\cdot h_-^t=\id,
\end{split}
\]
where we used the identities~\eqref{eq:identities}.

As $0<s<q$, we have that $(-1/2,s/2q) \in S$ and $(-1/2,s/2q-1/2) \notin S$, so we have
\[
	(h_-)_*\left(-\frac{1}{2},\frac{s}{2q}\right) = h_-, 	\qquad \text{and} \qquad	(h_+)_*\left(-\frac{1}{2},\frac{s}{2q}-\frac{1}{2}\right)=h_+^{-1}.
\]
Moreover
\[
	(h_+)_*\left(\frac{s}{2q}, -\frac{1}{2}\right) = h_+,	\qquad \text{and} \qquad	(h_-)_*\left(\frac{s}{2q}-\frac{1}{2},-\frac{1}{2}\right)=h_-^{-1}.
\]
Carrying our analysis further, we have,
\[
	(h_-^{a-1})_* \left(\frac{s}{2q},\frac{s}{2q}-\frac{1}{2}\right) = h_-^{ a-1 -2(k-1)} = h_-^{a-2k+1}=h_-^{\tilde{a}+1}.
\]
Similarly
\[
	(h_+^{d-1})_* \left(\frac{s}{2q}-\frac{1}{2}, \frac{s}{2q}\right) = h_+^{d-1-2(l-1)} = h_+^{\tilde{d}+1}.
\]

To conclude, we only have to show that
\[
	(h_+^n)_* \left(\frac{s}{2q},-\frac{r}{2q}\right) = h_+^t.
\]
This is clear since $(s-8mqr)/2q= s/2q-4mr$ and thus Corollary~\ref{thm:homologyaction} gives
\[
	(h_+^n)_*\left(\frac{s}{2q},-\frac{r}{2q}\right) = h_+^{n-2\abs{4mr}} = h_+^t.
\]
This completes the proof of the Lemma.
\end{proof}

To get a lower bound on the Hausdorff dimension on the set of ergodic directions we will use the following standard result.

\begin{prop}\label{thm:Hausdorffdimension}
For any $\bar{a},\bar{b}\in\NN^m$, with $m\geq0$, and for any set $D=d\NN+c\subset\NN$, with $d,c\in\NN$ and $d>0$, $c\geq0$, the Hausdorff dimension of the set
\[
	\mathcal{E}(\bar{a},\bar{b})=\Set{[0;\bar{a},n_1,\bar{b},n_1,\bar{a},n_2,\bar{b},n_2,\dots]: n_i\in D}
\]
is greater than $1/2$.
\end{prop}

\begin{proof}
Write $\bar{a}=a_1\dots a_m$ and $\bar{b}=b_1\dots b_m$.
We can assume $m\geq4$ and $m$ even.
For $l\in\NN$, define the map $\psi_{\bar{a},\bar{b},l}\colon[0,1]\to[0,1]$ by
\[
	\psi_{\bar{a},\bar{b},l}(x)=[0;a_1,\dots,a_m,l,b_1,\dots,b_m,l+x]=\frac{p_m(\bar{a},\bar{b},l)(l+x)+p_{m-1}(\bar{a},\bar{b},l)}{q_m(\bar{a},\bar{b},l)(l+x)+q_{m-1}(\bar{a},\bar{b},l)}.
\]
Thus
\[
	\psi_{\bar{a},\bar{b},l}([0,1])=\bigl[ [0;a_1,\dots,a_m,l,b_1,\dots,b_m,l], [0;a_1,\dots,a_m,l,b_1,\dots,b_m,l+1] \bigr].
\]
For $x\in[0,1]$, omitting the dependence of $q_m$ and $q_{m-1}$ by $\bar{a}$, $\bar{b}$ and $l$, one has
\[
	\psi'_{\bar{a},\bar{b},l}(x)=\frac{1}{(q_m(l+x)+q_{m-1})^2}\geq\frac{1}{(q_m(l+1)+q_{m-1})^2}=:e_{\bar{a},\bar{b},l}
\]
and
\[
	\psi'_{\bar{a},\bar{b},l}(x)\leq\frac{1}{(q_ml+q_{m-1})^2}<\frac{1}{4}.
\]
Now, for every $u\in\NN$, let $D_u=d\{1,\dots,u\}+c$.
Call
\[
	\mathcal{E}_u(\bar{a},\bar{b})=\bigcap_{k\geq1}\bigcup_{(n_1,\dots,n_k)\in(D_u)^k}
							\psi_{\bar{a},\bar{b},n_1}\circ\psi_{\bar{a},\bar{b},n_2}\circ\dots\circ\psi_{\bar{a},\bar{b},n_k}[0,1].
\]
Then $\mathcal{E}_u(\bar{a},\bar{b})\subset\mathcal{E}(\bar{a},\bar{b})$.

Let
\[
	E_u=\bigl[ [0;\bar{a},du+c,\bar{b},du+c], [0;\bar{a},d+c,\bar{b},d+c+1]\bigr].
\]
Remark that, since $m$ is even, $[0;\bar{a},du+c,\dots]<[0;\bar{a},d+c,\dots]$.
By definition of $\psi_{\bar{a},\bar{b},l}$ we have that $\psi_{\bar{a},\bar{b},l}(E_u)\subset E_u$ for all $l\in D_u$.
Moreover, the intervals $\psi_{\bar{a},\bar{b},l}(E_u)$ are pairwise disjoint.
Proposition 9.7 of~\cite{Falconer:FractalGeometry} assures that the Hausdorff dimension $\dim_H(\mathcal{E}_u(\bar{a},\bar{b}))\geq s_u$, where $s_u>0$ is the unique solution to the equation 
\[
	\sum_{l=1}^u e_{\bar{a},\bar{b},dl+c}^{s_u}=1.
\]
Since $\sum_{l=1}^\infty e_{\bar{a},\bar{b},dl+c}^{1/2}=+\infty$, we can find a $u\in\NN$ such that $\sum_{l=1}^u e_{\bar{a},\bar{b},dl+c}^{1/2}>1$.
In particular $s_u>1/2$ and so
\[
	\dim_H(\mathcal{E}(\bar{a},\bar{b}))\geq\dim_H(\mathcal{E}_u(\bar{a},\bar{b}))\geq s_u >1/2,
\]
as we wanted to show.
\end{proof}

\section{Results for $\widetilde{M}(z)_\Gamma$}\label{sec:proofMR}

\begin{thm}\label{thm:rationaldirections}
Suppose that $z=(r/2q,s/2q)\in \TT_0^2$, where $r$, $s$ and $q$ are integer numbers with $\abs{r},\abs{s}<q$ and $s\neq0$ is coprime with $q$.
Suppose moreover that at least one number between $r$ and $s$ is odd.
For every sequence of natural numbers $(n_i)_{i\geq1}$ in $8q\NN$, set
\[
	\vartheta=[0;d-1,1,1,d,n_1,a-1,1,1,a,n_1,d-1,1,1,d,n_2,a-1,\dots].
\]
Then the directional flow along the direction $\theta$ of the vector $(1,\vartheta)$ on the $\ZZ^2$-cover $\widetilde{M}(z)_\Gamma$ given by $\Gamma=(\beta,-\alpha)$ is \emph{ergodic}.
Furthermore, the Hausdorff dimension of the set of such directions is bigger than $1/2$.
\end{thm}

\begin{proof}
We begin by noting that, grouping the entries of $\vartheta$ in blocks of length ten, we obtain the transformations $g_z(n_1)$, $g_z(n_2),\dots$ of the form described in Lemma~\ref{thm:lemmaodd}.
We then know that
\[
	g_z(n_i)\cdot M(z) = M(z), \qquad \text{and} \qquad (g_z(n_i))_*(z)=\id,
\]
for all $i$.
Writing, as usual, $\vartheta=[0;a_1,a_2,\dots]$ we have
\[
	a_{10k}\geq 8q \geq 4q+4\geq\frac{4q+4}{q-\abs{r}}=\frac{4(1+\epsilon)}{1-\frac{\abs{r}}{q}},
\]
with $\epsilon=\frac{1}{q}$.
Similarly
\[
	a_{10k+1}\geq 4q \geq 2q+2 \geq\frac{2q+2}{q-\abs{s}}=\frac{2(1+\epsilon)}{1-\frac{\abs{s}}{q}}.
\]
Thus all the hypothesis of Theorem~\ref{thm:conditionaltheorem} hold.
So the flow in direction $\theta$ of $(1,\vartheta)$ is ergodic on the infinite surface $\widetilde{M}(z)_\Gamma$.
Finally, the lower bound on the Hausdorff dimension of the set of ergodic directions is given by Proposition~\ref{thm:Hausdorffdimension}.
\end{proof}

To deduce a statement about topological abundance of ergodic directions in $\widetilde{M}(z)_\Gamma$ we adapt to our setting some results developed for $\ZZ$-covers.
If $z$ is rational, Corollary~5.7 of~\cite{HooperWeiss} assure us that the Veech group of $\widetilde{M}(z)_\Gamma$ is a discrete subgroup of $\PSL(2,\RR)$ whose limit set is $\RR\PP^1$.
Recall that a set is a $G_\delta$ if it can be obtained as a countable intersection of open sets.
Then, using the same strategy of Proposition~15 of~\cite{HubertWeiss} we prove the following

\begin{lemma}\label{thm:Gdeltasurface}
Suppose that $z=(r/2q,s/2q)\in \TT_0^2$, where $r$, $s$ and $q$ are integer numbers with $\abs{r},\abs{s}<q$ and $s\neq0$ is coprime with $q$.
Suppose moreover that at least one number $r$ or $s$ is odd.
Then the set of ergodic directions on $\widetilde{M}(z)_\Gamma$ forms a $G_\delta$ dense subset of $S^1$.
\end{lemma}

\begin{proof}
We have shown in the proof of  Theorem~\ref{thm:rationaldirections} that there exist two different strips $\widetilde{C}^h$ and $\widetilde{C}^v$ on $\widetilde{M}(z)_\Gamma$ with $k(C^h)=(1,0)$ and $k(C^v)=(0,1)$.
Since the Veech group of $\widetilde{M}(z)_\Gamma$ is discrete, we can enumerate its elements: $G=\{g_1,g_2,\dots\}$.
For a natural $n$ and a fixed $\epsilon>0$, call $B_n\subset S^1$ the set of $\theta$'s for which
\begin{equation}\label{eq:dwellapproximable}
	 \abs{(\cos\theta,\sin\theta) \wedge g v(C^h)} \leq (1-\epsilon) \frac{A(C^h)}{2},
		\quad \text{for some $g\in G\setminus\{g_1,\dots,g_n\}$.}
\end{equation}
Each of these sets is open.
Moreover, if $\theta\in B_n$ then all its orbit under $G$, except eventually $g_i g^{-1}$, for $i=1,\dots,n$, is contained in $B_n$ as well.
Since the limit set of $G$ is $\RR\PP^1$, this implies that each $B_n$ is dense in $S^1$.
Consider the set $B^h=\cap B_n$.
If a direction is in $B^h$ then it satisfies~\eqref{eq:dwellapproximable} for infinitely many different $g$'s.
It is thus well approximated by the infinite strips $g \widetilde{C}^h$.
Since the intersection form is invariant under diffeomorphisms, we have $k(g\widetilde{C}^h)=(1,0)$.
Summing up, we have shown that we can find a sequence of strips that produce the essential value $(1,0)$ and well approximate every direction in a $G_\delta$ dense set.

Reasoning in the same way on $C^v$, we get the existence of a $G_\delta$ dense subset of  $B^v\subset S^1$, for which similar conclusions hold for the essential value $(0,1)$.
The intersection $B^h\cap B^v$ is thus a $G_\delta$ dense subset of $S^1$ formed of directions well approximated by infinitely many strips that produce essential values $(1,0)$ and $(0,1)$.

To conclude, it suffices to remark that the directions in $B^h\cap B^v$ are ergodic directions for the directional flow on the compact surfaces $M(z)$.
This is true because they do not correspond to a cylinder decomposition of the surface.
Since $M(z)$ is a square-tiled surface, and hence a Veech surface, they have to be ergodic directions.
\end{proof}

\section{Back to Eaton lenses}\label{sec:returntoeaton}
In this section we translate back the results we obtained in the previous sections to our original setting of systems of Eaton lenses in order to prove Theorem~\ref{thm:HausdorffEaton}.

Recall that when we reduced from the system $L(\Lambda,R)$ of $\Lambda$-periodic Eaton lenses of radius $R$ to its flat counterpart $F(\Lambda,R)$, we were interested only in the vertical direction and thus we could substitute a circular lens with its horizontal diameter.
In the proof of Theorem~\ref{thm:HausdorffEaton} we will use several times the following fact.

\begin{lemma}[Lemma~9.2 of~\cite{FraczekUlcigrai:notergodic}]\label{thm:ergodictoergodic}
Let $(M,\omega)$ be a translation surface and $\theta\in S^1$ a direction.
If $g\in\SL(2,\RR)$, call $\theta'$ the direction determined by $e^{i\theta'}=ge^{i\theta}/\norm{ge^{i\theta}}$.
Then there exists an $s>0$ such that the flows $\phi_{st}^{\theta'}$ on $g\cdot(M,\omega)$ and $\flow$ on $(M,\omega)$ are measure theoretically isomorphic via a homeomorphism.
In particular, one is ergodic if and only if the other is ergodic.
\end{lemma}

\begin{proof}[Proof of Theorem~\ref{thm:HausdorffEaton}]
Let $0<R<1/2$ be the radius of each lens.
Then $\ZZ^2$ is an $R$-admissible lattice.
Since $R$-admissibility is an open condition which is invariant under rotations, there exists an $\epsilon>0$ such that $G_t h_\tau r_\theta \cdot\ZZ^2$ is still $R$-admissible for $(t,\tau,\theta)\in (-\epsilon,\epsilon)^2\times S^1$ and
\[
	h_\tau=\begin{pmatrix}
			1	&	0\\
			\tau	&	1
		\end{pmatrix}.
\]

Remark that rational numbers of the form $s/2q$ with $0<s<q$ odd are dense in the interval $\bigl(0,\frac{1}{2}\bigr)$.
For $s/2q \to R$ and $q\to\infty$ one has
\[
	\frac{R}{\frac{s}{2q}\cos(\operatorname{arccot}4q)}\to 1.
\]
We can then find, for any given $\epsilon>0$, a $q$ big enough and a number $s/2q$ close enough to $R$ so that
\[
	\log\frac{R}{\frac{s}{2q}\cos(\operatorname{arccot}4q)}<\epsilon.
\]
Write $z=(0,\frac{s}{2q})$ and consider the surface $\widetilde{M}(z)_\Gamma$, where $\Gamma=(\beta,-\alpha)$. 
Theorem~\ref{thm:rationaldirections} gives the existence of a set of directions $E=E(s/2q)\subset S^1$, with $\dim_H E>1/2$, such that the flow in these directions is ergodic on the infinite surface $\widetilde{M}(z)_\Gamma$.
Since an ergodic direction $\theta$ is the direction of the vector $(1,\vartheta)$ and $\vartheta=[0;d-1,1,\dots]$, with $4q<d\leq6q$, we have $\cot\theta=[d-1;1,\dots]$.
As both $\cos$ and $\cot$ are decreasing functions in $(0,\pi)$ we have $\cos\theta\geq\cos(\operatorname{arccot}4q)$.
So, for all $\theta\in E$, we have
\begin{equation}\label{eq:geodesicadmissibility}
	t^*:=\log{\frac{R}{\frac{s}{2q}\cos\theta}}\leq\log{\frac{R}{\frac{s}{2q}\cos(\operatorname{arccot}4q)}}<\epsilon.
\end{equation}

Choose now one specific ergodic direction $\theta$ on $\widetilde{M}(z)_\Gamma$.
We are going to translate the ergodicity of the flow $\tflow$ on this surface into the ergodicity of the vertical flow in some periodic configuration of Eaton lenses exploiting the homogeneity of $\sL$.
The whole procedure is schematically represented in Figure~\ref{fig:slitnormalisation}.

We first apply a rotation of angle $\frac{\pi}{2}-\theta$ in anticlockwise direction in order to bring the direction $\theta$ on the vertical one.
We obtain the surface $r_{\frac{\pi}{2}-\theta}\cdot\widetilde{M}(z)_\Gamma$.
Remark that the vertical flow on this infinite surface is ergodic.

We can now apply the horocycle flow $h_\tau$ for some time $\tau\in(-\epsilon,\epsilon)$.
As $h_\tau$ fixes all the vertical vectors, Lemma~\ref{thm:ergodictoergodic} guarantees that the vertical flow is still ergodic on the infinite surface $h_\tau r_{\frac{\pi}{2}-\theta}\cdot\widetilde{M}(z)_\Gamma$.
Moreover, the lattice $h_\tau r_{\frac{\pi}{2}-\theta}\cdot\ZZ^2$ is $R$-admissible by our assumption on $\tau$.

Apply the geodesic flow $G_t$ for time $t^*$ given by~\eqref{eq:geodesicadmissibility}.
We obtain the infinite surface $G_{t^*} h_\tau r_{\frac{\pi}{2}-\theta}\cdot\widetilde{M}(z)_\Gamma$ on which the vertical flow is ergodic, once again by Lemma~\ref{thm:ergodictoergodic}.
Finally, we can project the slit we have obtained onto a horizontal one centred at the same point.
Time $t^*$ is chosen so that the resulting slit has precisely length $2R$.
Calling $\Lambda = G_{t^*} h_\tau r_{\frac{\pi}{2}-\theta}\cdot\ZZ^2$ and using the notation of the introduction, we have obtained the surface $\widetilde{M}(\Lambda,R)$, the orientation covering of $F(\Lambda,R)$, the $\Lambda$-periodic configuration of flat lenses of length $2R$.
The vertical flow on this new infinite surface has the same global behaviour of the vertical flow on $G_{t^*} h_\tau r_{\frac{\pi}{2}-\theta}\cdot\widetilde{M}(z)_\Gamma$ since the two flows differ only on a small neighbourhood of the slits.
This tells us, in particular, that the vertical flow on the surface $\widetilde{M}(\Lambda,R)$ is still ergodic.
Remark that the lattice $\Lambda$ is $R$-admissible thanks to~\eqref{eq:geodesicadmissibility}.
Hence, this result for the orientation covering of the flat lenses immediately implies that the Eaton flow on the corresponding system of circular Eaton lenses is ergodic.

\begin{figure}[t]
\centering
\def\svgwidth{0.9\textwidth}
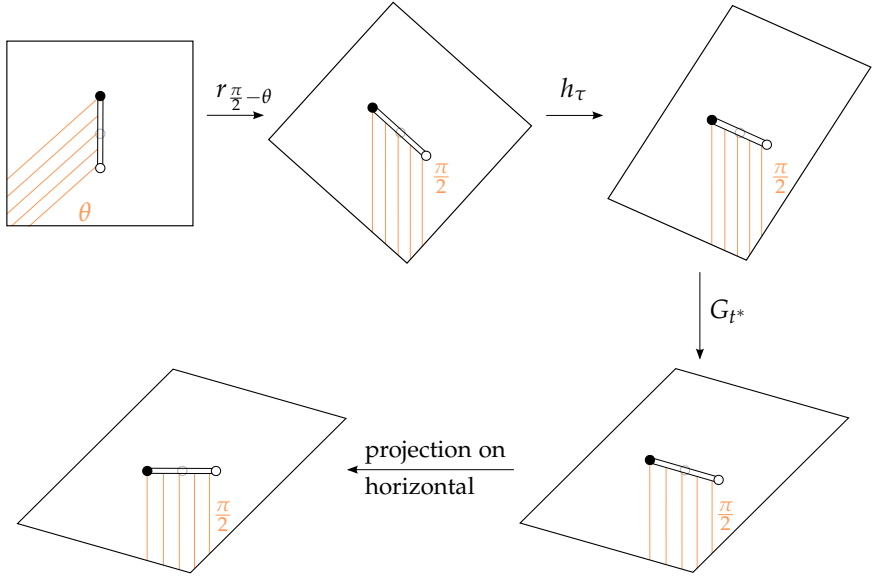
\caption{A cartoon of the normalisation of the slit as described in the proof of Theorem~\ref{thm:HausdorffEaton}.}
\label{fig:slitnormalisation}
\end{figure}

Consider now the set
\[
	\mathscr{E}=\Set{G_{t^*}h_\tau r_{\frac{\pi}{2}-\theta} \cdot\ZZ^2, \tau\in(-\epsilon,\epsilon), \theta\in E\left(\tfrac{s}{2q}\right), t^*=\log{\frac{R}{\frac{s}{2q}\cos\theta}}}\subset\sL.
\]
We just showed that these lattices are $R$-admissible and that, for every $\Lambda\in\mathscr{E}$, the vertical flow on $\widetilde{M}(\Lambda,R)$ is ergodic.
We now want to estimate the Hausdorff dimension of this set.

The local product structure on $\sL$ given by the Iwasawa ANK decomposition of $\SL(2,\RR)$ implies that the Haar measure $\mu_\sL$ is locally equivalent to the product Lebesgue measure in the coordinates $(t,\tau,\theta)$.
Since the $t$-coordinate is uniquely determined in terms of $\theta$, $\dim_H \mathscr{E} = \dim_H \pi_{2,3}(\mathscr{E})$, where
\[
\pi_{2,3}(\mathscr{E})=\Set{h_\tau r_{\frac{\pi}{2}-\theta} \cdot\ZZ^2, \tau\in(-\epsilon,\epsilon), \theta\in E\left(\tfrac{s}{2q}\right)}\subset\sL.
\]
Writing $\pi_{2,3}(\mathscr{E})=\bigcup_{\tau\in (-\epsilon,\epsilon)} h_\tau \cdot \Set{r_{\frac{\pi}{2}-\theta} \cdot\ZZ^2 \theta\in E\left(\tfrac{s}{2q}\right)}$, we see that, in the $(\tau,\theta)$ coordinates, $\pi_{2,3}(\mathscr{E})=(-\epsilon,\epsilon) \times E$.
This tells us that
\[
	\dim_H \mathscr{E} = \dim_H \pi_{2,3}(\mathscr{E})\geq1+\dim_H E > 1+\frac{1}{2}=\frac{3}{2},
\]
as we wanted to prove.
\end{proof}

\section*{Acknowledgments}
The author thanks his supervisor Corinna Ulcigrai for her invaluable support and guidance throughout the writing of this paper.
We thank Krzysztof Fr\k{a}czek for suggesting the problem to us and for useful discussions and Thomas Jordan for the precious help on the results concerning Hausdorff dimension.
We also thank the anonymous referee for his/her suggestions that improved the presentation of the paper.

%\clearpage
%\phantomsection
%\addcontentsline{toc}{section}{\refname}
\printbibliography
\end{document}

%% file: morelenses.pdf_tex
%% Creator: Inkscape 0.48.3.1, www.inkscape.org
%% Accompanies image file 'morelenses.pdf' (pdf, eps, ps)
\begingroup%
  \makeatletter%
  \providecommand\color[2][]{%
    \errmessage{(Inkscape) Color is used for the text in Inkscape, but the package 'color.sty' is not loaded}%
    \renewcommand\color[2][]{}%
  }%
  \providecommand\transparent[1]{%
    \errmessage{(Inkscape) Transparency is used (non-zero) for the text in Inkscape, but the package 'transparent.sty' is not loaded}%
    \renewcommand\transparent[1]{}%
  }%
  \providecommand\rotatebox[2]{#2}%
  \ifx\svgwidth\undefined%
    \setlength{\unitlength}{280.8bp}%
    \ifx\svgscale\undefined%
      \relax%
    \else%
      \setlength{\unitlength}{\unitlength * \real{\svgscale}}%
    \fi%
  \else%
    \setlength{\unitlength}{\svgwidth}%
  \fi%
  \global\let\svgwidth\undefined%
  \global\let\svgscale\undefined%
  \makeatother%
  \begin{picture}(1,0.4881557)%
    \put(0,0){\includegraphics[width=\unitlength]{morelenses.pdf}}%
    \put(0.71652422,0.00243723){\color[rgb]{0,0,0}\makebox(0,0)[lb]{\smash{$F(\Lambda,R)$}}}%
    \put(0.18660969,0.00243723){\color[rgb]{0,0,0}\makebox(0,0)[lb]{\smash{$L(\Lambda,R)$}}}%
  \end{picture}%
\endgroup%

%% file: flatlenses.pdf_tex
%% Creator: Inkscape 0.48.3.1, www.inkscape.org
%% Accompanies image file 'flatlenses.pdf' (pdf, eps, ps)
\begingroup%
  \makeatletter%
  \providecommand\color[2][]{%
    \errmessage{(Inkscape) Color is used for the text in Inkscape, but the package 'color.sty' is not loaded}%
    \renewcommand\color[2][]{}%
  }%
  \providecommand\transparent[1]{%
    \errmessage{(Inkscape) Transparency is used (non-zero) for the text in Inkscape, but the package 'transparent.sty' is not loaded}%
    \renewcommand\transparent[1]{}%
  }%
  \providecommand\rotatebox[2]{#2}%
  \ifx\svgwidth\undefined%
    \setlength{\unitlength}{437.875bp}%
    \ifx\svgscale\undefined%
      \relax%
    \else%
      \setlength{\unitlength}{\unitlength * \real{\svgscale}}%
    \fi%
  \else%
    \setlength{\unitlength}{\svgwidth}%
  \fi%
  \global\let\svgwidth\undefined%
  \global\let\svgscale\undefined%
  \makeatother%
  \begin{picture}(1,0.31374335)%
    \put(0,0){\includegraphics[width=\unitlength]{flatlenses.pdf}}%
    \put(0.1013988,0.00157722){\color[rgb]{0,0,0}\makebox(0,0)[lb]{\smash{$F(\Lambda,R)$}}}%
    \put(0.48689695,0.00157722){\color[rgb]{0,0,0}\makebox(0,0)[lb]{\smash{$F_+(\Lambda,R)$}}}%
    \put(0.80662289,0.00157722){\color[rgb]{0,0,0}\makebox(0,0)[lb]{\smash{$F_-(\Lambda,R)$}}}%
  \end{picture}%
\endgroup%

%% file: mz.pdf_tex
%% Creator: Inkscape 0.48.3.1, www.inkscape.org
%% Accompanies image file 'mz.pdf' (pdf, eps, ps)
\begingroup%
  \makeatletter%
  \providecommand\color[2][]{%
    \errmessage{(Inkscape) Color is used for the text in Inkscape, but the package 'color.sty' is not loaded}%
    \renewcommand\color[2][]{}%
  }%
  \providecommand\transparent[1]{%
    \errmessage{(Inkscape) Transparency is used (non-zero) for the text in Inkscape, but the package 'transparent.sty' is not loaded}%
    \renewcommand\transparent[1]{}%
  }%
  \providecommand\rotatebox[2]{#2}%
  \ifx\svgwidth\undefined%
    \setlength{\unitlength}{325.48671875bp}%
    \ifx\svgscale\undefined%
      \relax%
    \else%
      \setlength{\unitlength}{\unitlength * \real{\svgscale}}%
    \fi%
  \else%
    \setlength{\unitlength}{\svgwidth}%
  \fi%
  \global\let\svgwidth\undefined%
  \global\let\svgscale\undefined%
  \makeatother%
  \begin{picture}(1,0.44026824)%
    \put(0,0){\includegraphics[width=\unitlength]{mz.pdf}}%
    \put(0.19114697,0.27064771){\color[rgb]{0,0,0}\rotatebox{-35.00000031}{\makebox(0,0)[lb]{\smash{$R$}}}}%
    \put(0.25154774,0.22835463){\color[rgb]{0,0,0}\rotatebox{-35.00000031}{\makebox(0,0)[lb]{\smash{$L$}}}}%
    \put(0.75844864,0.27190539){\color[rgb]{0,0,0}\rotatebox{-35.00000031}{\makebox(0,0)[lb]{\smash{$L$}}}}%
    \put(0.81884941,0.22961232){\color[rgb]{0,0,0}\rotatebox{-35.00000031}{\makebox(0,0)[lb]{\smash{$R$}}}}%
    \put(0.2614138,0.00879212){\color[rgb]{0,0,0}\makebox(0,0)[lb]{\smash{$\alpha$}}}%
    \put(-0.00157697,0.29144572){\color[rgb]{0,0,0}\makebox(0,0)[lb]{\smash{$\beta$}}}%
    \put(0.72840674,0.00879212){\color[rgb]{0,0,0}\makebox(0,0)[lb]{\smash{$\alpha$}}}%
    \put(0.55635671,0.29144572){\color[rgb]{0,0,0}\makebox(0,0)[lb]{\smash{$\beta$}}}%
  \end{picture}%
\endgroup%

%% file: S.pdf_tex
%% Creator: Inkscape 0.48.3.1, www.inkscape.org
%% Accompanies image file 'S.pdf' (pdf, eps, ps)
\begingroup%
  \makeatletter%
  \providecommand\color[2][]{%
    \errmessage{(Inkscape) Color is used for the text in Inkscape, but the package 'color.sty' is not loaded}%
    \renewcommand\color[2][]{}%
  }%
  \providecommand\transparent[1]{%
    \errmessage{(Inkscape) Transparency is used (non-zero) for the text in Inkscape, but the package 'transparent.sty' is not loaded}%
    \renewcommand\transparent[1]{}%
  }%
  \providecommand\rotatebox[2]{#2}%
  \ifx\svgwidth\undefined%
    \setlength{\unitlength}{325.48671875bp}%
    \ifx\svgscale\undefined%
      \relax%
    \else%
      \setlength{\unitlength}{\unitlength * \real{\svgscale}}%
    \fi%
  \else%
    \setlength{\unitlength}{\svgwidth}%
  \fi%
  \global\let\svgwidth\undefined%
  \global\let\svgscale\undefined%
  \makeatother%
  \begin{picture}(1,0.44011663)%
    \put(0,0){\includegraphics[width=\unitlength]{S.pdf}}%
    \put(0.19114697,0.27064771){\color[rgb]{0,0,0}\rotatebox{-35.00000031}{\makebox(0,0)[lb]{\smash{$R$}}}}%
    \put(0.25154774,0.22835463){\color[rgb]{0,0,0}\rotatebox{-35.00000031}{\makebox(0,0)[lb]{\smash{$L$}}}}%
    \put(0.75844864,0.27190539){\color[rgb]{0,0,0}\rotatebox{-35.00000031}{\makebox(0,0)[lb]{\smash{$L$}}}}%
    \put(0.81884941,0.22961232){\color[rgb]{0,0,0}\rotatebox{-35.00000031}{\makebox(0,0)[lb]{\smash{$R$}}}}%
    \put(0.2614138,0.00879212){\color[rgb]{0,0,0}\makebox(0,0)[lb]{\smash{$\alpha$}}}%
    \put(-0.00157697,0.29144572){\color[rgb]{0,0,0}\makebox(0,0)[lb]{\smash{$\beta$}}}%
    \put(0.72840674,0.00879212){\color[rgb]{0,0,0}\makebox(0,0)[lb]{\smash{$\alpha$}}}%
    \put(0.55635671,0.29144572){\color[rgb]{0,0,0}\makebox(0,0)[lb]{\smash{$\beta$}}}%
  \end{picture}%
\endgroup%

%% file: slitnormalisation.pdf_tex
%% Creator: Inkscape 0.48.3.1, www.inkscape.org
%% Accompanies image file 'slitnormalisation.pdf' (pdf, eps, ps)
\begingroup%
  \makeatletter%
  \providecommand\color[2][]{%
    \errmessage{(Inkscape) Color is used for the text in Inkscape, but the package 'color.sty' is not loaded}%
    \renewcommand\color[2][]{}%
  }%
  \providecommand\transparent[1]{%
    \errmessage{(Inkscape) Transparency is used (non-zero) for the text in Inkscape, but the package 'transparent.sty' is not loaded}%
    \renewcommand\transparent[1]{}%
  }%
  \providecommand\rotatebox[2]{#2}%
  \ifx\svgwidth\undefined%
    \setlength{\unitlength}{744.12632095bp}%
    \ifx\svgscale\undefined%
      \relax%
    \else%
      \setlength{\unitlength}{\unitlength * \real{\svgscale}}%
    \fi%
  \else%
    \setlength{\unitlength}{\svgwidth}%
  \fi%
  \global\let\svgwidth\undefined%
  \global\let\svgscale\undefined%
  \makeatother%
  \begin{picture}(1,0.66590928)%
    \put(0,0){\includegraphics[width=\unitlength]{slitnormalisation.pdf}}%
    \put(0.08345357,0.41017064){\color[rgb]{1,0.6,0.33333333}\makebox(0,0)[lb]{\smash{$\theta$}}}%
    \put(0.49033967,0.45399194){\color[rgb]{1,0.6,0.33333333}\makebox(0,0)[lb]{\smash{$\frac{\pi}{2}$}}}%
    \put(0.88245587,0.4538319){\color[rgb]{1,0.6,0.33333333}\makebox(0,0)[lb]{\smash{$\frac{\pi}{2}$}}}%
    \put(0.24256634,0.55530729){\color[rgb]{0,0,0}\makebox(0,0)[lb]{\smash{$r_{\frac{\pi}{2}-\theta}$}}}%
    \put(0.64034827,0.55530729){\color[rgb]{0,0,0}\makebox(0,0)[lb]{\smash{$h_{\tau}$}}}%
    \put(0.81236207,0.29728658){\color[rgb]{0,0,0}\makebox(0,0)[lb]{\smash{$G_{t^*}$}}}%
    \put(0.81881259,0.0607676){\color[rgb]{1,0.6,0.33333333}\makebox(0,0)[lb]{\smash{$\frac{\pi}{2}$}}}%
    \put(0.23801324,0.06013826){\color[rgb]{1,0.6,0.33333333}\makebox(0,0)[lb]{\smash{$\frac{\pi}{2}$}}}%
    \put(0.41458015,0.13602364){\color[rgb]{0,0,0}\makebox(0,0)[lb]{\smash{{\small projection on}}}}%
    \put(0.41458015,0.09302019){\color[rgb]{0,0,0}\makebox(0,0)[lb]{\smash{{\small horizontal}}}}%
  \end{picture}%
\endgroup%